\documentclass{ieeetran}

\usepackage{graphicx,color}
\usepackage{latexsym}
\usepackage{amsmath,amssymb,amsfonts,amsthm,dsfont}
\usepackage{enumerate}
\usepackage{yhmath}
\usepackage{rotating}
\usepackage{booktabs}

\theoremstyle{commentStyle}

\theoremstyle{responseStyle}

\def\be{\begin{equation}}
\def\ee{\end{equation}}
\def\ba{\begin{array}}
\def\ea{\end{array}}
\def\smod{\color{black}}
\def\emod{\color{black}}

\def\emodc{\color{black}}

{

}
\newtheorem{assum}{Assumption}
\newtheorem{definition}{Definition}
\newtheorem{thm}{Theorem}

\newtheorem{lem}{Lemma}

\newtheorem{proposition}{Proposition}
\def\smod{\color{black}}
\def\emod{\color{black}}

\def\qedp{\hspace*{\fill}~{\tiny $\blacksquare$}}
\newtheorem{itremark}{Remark}
\newenvironment{remark}{\begin{itremark}\rm}{\end{itremark}}
\newenvironment{rem}{\begin{itremark}\rm}{\end{itremark}}

\newcounter{mytempeqncnt}

\usepackage{mathrsfs} 

\DeclareMathAlphabet{\pazocal}{OMS}{zplm}{m}{n}

\usepackage[prependcaption,colorinlistoftodos]{todonotes}

\DeclareMathOperator{\im}{im}
\DeclareMathOperator{\col}{col}

\newcommand{\calS}{\ensuremath{\mathcal{S}}}
\newcommand{\calH}{\ensuremath{\mathcal{H}}}
\newcommand{\calU}{\ensuremath{\mathcal{U}}}

\newcommand{\R}{\ensuremath{\mathbb R}}

\newcommand{\fpart}[2]{\displaystyle \frac{\partial #1}{\partial #2}}
\newcommand{\barfpart}[2] {\displaystyle \left.\frac{\partial #1}{\partial #2}\right|_-}
\newcommand{\barfpartT}[2] {\displaystyle \left.\frac{\partial #1}{\partial #2}\right|_-^T}

\newcommand{\new}[1]{\textcolor{black}{#1}}
\newcommand{\alert}[1]{\textcolor{black}{#1}}

\newcommand{\New}[1]{\textcolor{black}{#1}}
\newcommand{\nima}[1]{\textcolor{black}{#1}}
\newcommand{\claudio}[1]{\textcolor{black}{#1}}

\newcommand{\cter}[1]{\textcolor{black}{#1}}

\title{
{\color{black}Bregman storage functions} for microgrid control 
}
\author{C. De Persis and N.~Monshizadeh\thanks{C.~De Persis and N.~Monshizadeh are with 
ENTEG and the J.C.~Willems Center for Systems and Control, University of Groningen, the Netherlands. E-mail: \{c.de.persis, n.monshizadeh\}@rug.nl.}}

\begin{document}
\maketitle
\begin{abstract}
In this paper we contribute a theoretical framework that sheds a new light on the problem of microgrid analysis and control. The starting point is an energy function comprising the {\color{black}``}kinetic{\color{black}"} energy associated with the elements that emulate the rotating machinery and terms taking into account the reactive power stored in the lines and dissipated on shunt elements. We then shape this energy function with the addition of an adjustable voltage-dependent term, and construct {\color{black}so-called Bregman} storage functions satisfying suitable dissipation inequalities. Our choice of the voltage-dependent term depends on the voltage dynamics under investigation. Several microgrids dynamics that have similarities or coincide with dynamics already considered in the literature are captured in our incremental energy analysis framework.  The twist with respect to existing results is that our incremental storage functions allow for a large signal analysis of the coupled microgrid obviating the need for simplifying linearization techniques and for the restrictive decoupling assumption in which the frequency dynamics is fully separated from the voltage one. A complete Lyapunov stability analysis of the various systems is carried out along with a discussion on their active and reactive power sharing properties.
\end{abstract}

\section{Introduction}
Microgrids have been envisioned as one of the leading technologies to increase the penetration of renewable energies in the power market. A thorough discussion of the technological, physical and control-theoretic  aspects of microgrids is provided in many interesting comprehensive works, including \smod \cite{zhong, zhong2,green,lewis,schiffer.thesis}. \emod

Power electronics allows inverter in the microgrids to emulate desired dynamic behavior. This is an essential feature since when the microgrid is in grid forming mode, inverters have to inject active and reactive power in order to supply the loads in a shared manner and maintain the desired frequency and voltage values at the nodes. Hence, much work has focused on the design of dynamics for the inverters that achieve these desired properties and this effort has involved both practitioners and theorists,  all providing a myriad of solutions, whose performance has been tested mainly numerically and experimentally.

The main obstacle however remains a systematic design of the microgrid controllers that achieve the desired properties in terms of frequency and voltage regulation with power sharing. The difficulty lies in the complex structure of these systems, comprising dynamical models of inverters and loads that are physically interconnected via exchange of active and reactive power. In quasi steady state working conditions, these quantities are sinusoidal terms depending on the voltage phasor relative phases. As a result, mathematical models of microgrids reduce to high-order oscillators interconnected via sinusoidal coupling, where the coupling weights depend on the voltage magnitudes obeying additional  dynamics. The challenges with these models lie in the presence of highly nonlinear terms and the strict coupling between active and reactive power flow equations.

To deal with the aforementioned complexity of these dynamical models common remedies are to decouple frequency and voltage dynamics, and to linearize the power flow equations. While the former enables a separate analysis of the two dynamics (\cite{schiffer.tcst}), the latter permits the use of a small signal argument to infer stability results; see e.g. \cite{simpson.quadratic, simpson.industrial}.

Results that deal with the fully coupled system are also available 
\cite{schiffer.aut,wang.lemmon, muenz}. In this case, the results mainly concern network-reduced models with primary control, namely stability rather than stabilization of the equilibrium solution. Furthermore,  lossy transmission lines can also be studied \cite{dorfler.bullo.sicon, wang.lemmon, bouattour, wang.lemmon, muenz}, and also \cite{chopra}.  

\cter{{\it Main contribution}.} In spite of these many advances, what is still missing is a comprehensive approach to deal with the analysis and control design for microgrids. In this paper we provide a contribution in this direction. The starting point is the energy function associated with the system, a combination of kinetic and potential energy. Relying on an extended notion of incremental dissipativity, {\color{black} a number of so-called Bregman storage functions} whose critical points have desired features are constructed. \smod The construction is inspired by works in the control of networks in the presence of disturbances, which makes use of  internal model controllers (\cite{mathias.claudio, Nima-Claudio-15}) and  incremental passivity (\cite{stan.sepulchre.tac07}).  The {\color{black} storage} functions that we design {encompass  several} {\color{black} network-reduced versions of} microgrid dynamics that have appeared in the literature, including the conventional droop controller \cite{zhong, schiffer.aut}, the quadratic droop controller \cite{simpson.quadratic}, and the  reactive power consensus dynamics \cite{schiffer.tcst}. Our analysis, however, suggests suitable modifications such as an  exponential  scaling of the averaging  reactive power dynamics of \cite{schiffer.tcst},  and inspires new controllers, such as the so-called reactive current controller  \cter{(we refer to \cite{Brabandere} for a related controller)}.  The approach we propose  has two additional distinguishing features: we do not need to assume decoupled dynamics and we perform a large signal analysis. 

 Our contribution also expands the knowledge on the use of energy functions in the context of microgrids. Although historically energy functions  have played a crucial role to deal with accurate models of power systems (\cite{tsolas,chu.chiang.cssp1999,chiang.et.al.proc.ieee95}), 
our approach based on the incremental dissipativity notion sheds a new light into the construction of these energy functions, allows us to cover a wider range of microgrid dynamics, and paves the way for the design of dynamic controllers, following the combination of passivity techniques and internal model principles as in \cite{mathias.claudio}. \smod We refer the reader to 
e.g.~\cite{ortega.et.al.tac05,dib.et.al.tac09} for seminal work on passivity-based control of power networks. \emod

In this paper we focus on network reduced models of microgrids (\cite{schiffer.aut,wang.lemmon,muenz,simpson.industrial}). These models are typically criticized for not providing an explicit characterization of the loads (\cite{simpson.quadratic}). Focusing on network reduced models allows us to reduce the technical complexity of the arguments and to provide an elegant analysis. However, one of the advantages of the use of the energy functions is that they remain effective also with network preserved models (\cite{tsolas}). In fact, a preliminary investigation  not reported in this manuscript for the sake of brevity shows that the presented results extend to the case of network preserved models. A full investigation of this case will be reported elsewhere. 

The outline of the paper is as follows. In Section \ref{sec.II},  details on the model under consideration are provided. In Section  \ref{sec.III} the design of \claudio{Bregman storage} functions is carried out and incremental dissipativity of various models of microgrids associated with different voltage dynamics is shown. A few technical conditions on these energy functions are discussed in Section  \ref{sec.IV}, and a {\color{black}distributed} test to check them is also provided. Based on the results of these sections, attractivity of the prescribed synchronous solution and voltage stability is presented in Section  \ref{sec.attractivity}, along with a discussion on power sharing properties of the proposed controllers \claudio{(Subsection \ref{subset.power.sharing}).} Power sharing in the presence of homogeneous  lossy  transmission lines  
\claudio{is studied in Subsection \ref{subs.lossy}}. \claudio{The paper ends with concluding remarks in Section \ref{sec.VI}.} 


\section{Microgrid model and a synchronous solution} \label{sec.II}
We consider a network-reduced model of a microgrid {\color{black} operating in islanded mode, that is disconnected from the main grid.
This model is given by} 
\be\label{micro.grid}
\ba{rcl}
\dot \theta &=& \omega\\
T_P\dot \omega &=& -(\omega-\omega^*) - K_P(P-P^*) +u_P\\ 
T_Q\dot V &=& f(V, Q, u_Q)
\ea\ee
where $\theta\in \mathbb{T}^n$ is the vector of voltage angles, $\omega\in \mathbb{R}^n$ is the frequency, $P\in \mathbb{R}^n$ is the active power vector, $Q\in \mathbb{R}^n$ is the reactive power vector, and $V\in \mathbb{R}^n_{>0}$ is the vector of voltage magnitudes. \cter{The integer $n$ equals the number of nodes in   the microgrid and $\mathcal{I}:=\{1,2,\ldots, n\}$ is the set of indices associated with the nodes.} The matrices $T_P$, $T_V$, and $K_P$ are diagonal and positive definite.
The vectors $\omega^\ast$ and $P^\ast$ denote the frequency and active power setpoints, respectively. \new{The vector   $P^\ast$ \cter{may} also models active power loads at  the buses {\color{black}(see Remark \ref{r:loads})}}.  The vector $u_Q$ is an additional input. The function $f$ accounts for the voltage dynamics/controller and is decided later. 

The model \eqref{micro.grid} with an appropriate selection of $f$ {\color{black}describes} various models of network-reduced microgrids in the literature, including conventional droop controllers, quadratic droop controllers, and consensus based reactive power control schemes (\cite{zhong,simpson,schiffer.aut,simpson.quadratic,schiffer.tcst}). 
{\color{black}However,  while \cite{simpson,simpson.quadratic,schiffer.modelling} consider network-preserved models of microgrids, in this paper network-reduced models are considered. }
\smod We refer the reader to \cite{schiffer.modelling} for a compelling derivation of microgrid models from first principles. 
\emod

Our goal here is to provide a unifying framework for analysis of the microgrid model \eqref{micro.grid} for different types of voltage controllers, and  study frequency regulation, voltage stability, and active as well as reactive power sharing. A key point of our approach is that it does not rely on simplifying and often restrictive premises such as the decoupling assumption and linear approximations.  


\smallskip

\noindent {\it Active and reactive power.} 
The active power $P_i$ is given by
\be\label{active}
P_i = \sum_{j\in \mathcal{N}_i} B_{ij} V_iV_j \sin \theta_{ij},\quad \theta_{ij}:= \theta_i-\theta_j
\ee
and the reactive power by
\be\label{reactive}
Q_i = B_{ii} V_i^2-\sum_{j\in \mathcal{N}_i} B_{ij} V_iV_j \cos \theta_{ij},\quad \theta_{ij}:= \theta_i-\theta_j.
\ee
Note that here 
$B_{ii}= \hat B_{ii}+\sum_{j\in \mathcal{N}_i} B_{ij}$, where $B_{ij}=B_{ji}>0$ is the negative of the susceptance at edge $\{i,j\}$ and $\hat B_{ii}\ge 0$ is the negative of the  shunt susceptance at node $i$.\footnote{{\color{black} See Remark \ref{r:loads} for a discussion on the physical meaning of these shunt susceptances.}}
Hence, 
$B_{ii}\ge \sum_{j\in \mathcal{N}_i} B_{ij}$ for all $i$. 

It is useful to have compact representations of both active and reactive power. 
Setting 
$\Gamma(V)={\rm diag}(\gamma_1(V),\ldots, \gamma_{\color{black}m}(V))$, $\gamma_k(V)= V_i V_j B_{ij}$, with $k\in  E:=  \{1,2,\ldots,{\color{black}m}\}$  being the index corresponding to the edge $\{i,j\}$ (in short, $k\sim \{i,j\}$), 
the vector of the active power at all the nodes writes as
\[
P= D\Gamma(V) \boldsymbol{\sin} (D^T \theta). 
\]
where $D=[d_{ik}]$ is the incidence matrix of the graph describing the interconnection structure of the network, {and the \alert{vector} $\boldsymbol{\sin}(\cdot)$ is defined element-wise.} Let us now introduce the vector $A_0= {\rm col} (B_{11},\ldots, B_{nn})$. Since $|d_{ik}| \cos (d_{ik}\theta_i +d_{jk}\theta_j)= \cos(\theta_i-\theta_j)$, {\color{black}for $k\sim\{i,j\}$,} the vector of reactive power at the nodes takes the form
\[
Q= [V][A_0] V - |D| \Gamma(V) \boldsymbol{\cos}(D^T \theta), 
\]
where $|D|$ is obtained by replacing each element $d_{ij}$ of $D$ with $|d_{ij}|$.\footnote{In fact, denoted by $\eta$ the vector $D^T \theta$, the entry $ij$ of the matrix $|D| \Gamma(V) \boldsymbol{\cos}(D^T \theta)$ writes as 
\[\ba{rcl}
[|D| \Gamma(V) \boldsymbol{\cos}(D^T \theta)]_{ij}&=&\displaystyle \sum_{k=1}^{\cter{m}} |d_{ik}| \gamma_k(V) \cos(\eta_k)\\[2mm]
&=&\displaystyle  \sum_{k\sim\{i,j\}} |d_{ik}| V_iV_j B_{ij} \cos(d_{ik}\theta_i +d_{jk}\theta_j)\\[2mm]
&=& \displaystyle \sum_{j\in \mathcal{N}_i} V_iV_j B_{ij} \cos(\theta_i -\theta_j)
\ea\]
}
Moreover, here and throughout the paper, the notation $[v]$ represents the diagonal matrix associated with vector $v$. 

Another compact representation is useful as well. To this end, introduce the symmetric matrix 
	\[
	\ba{l}
	\mathcal{A} (\boldsymbol{\cos}(D^T \theta)) = \\
	\begin{bmatrix}
	B_{11} & -B_{12} \cos \theta_{12} & \ldots & -B_{1n} \cos \theta_{1n} \\
	-B_{21} \cos \theta_{21} & B_{22} & \ldots & -B_{2n} \cos \theta_{2n} \\
	\vdots & \vdots & \ddots & \vdots\\
	-B_{n1} \cos \theta_{n1}   & -B_{n2} \cos \theta_{n2} & \ldots & B_{nn}\\
	\end{bmatrix}.
	\ea
	\]
The vector $Q$ becomes 
\be\label{Q.A}\ba{rcl}
Q 
&= & [V] \mathcal{A}(\boldsymbol{\cos}(D^T \theta)) V,
\ea\ee
where again we are exploiting the identity  $\cos (d_{ik}\theta_i +d_{jk}\theta_j)= \cos
\cter{\theta_{ij}}$.

As a consequence of the condition $B_{ii}\ge \sum_{j\in \mathcal{N}_i} B_{ij}$ for all $i$, 
provided that 
at least one $\hat B_{ii}$ is non-zero (which is the standing assumption throughout the paper),
the symmetric matrix $\mathcal{A}(\boldsymbol{\cos}(D^T \theta))$ has all strictly positive eigenvalues and hence is a positive definite matrix.   
Note that the matrix 
$\mathcal{A}$ can be interpreted as  a loopy Laplacian matrix of the graph.  
\smallskip

{\color{black} 
Before proceeding further,  we remark on the adopted model.
\begin{rem} \label{r:lines}{\bf (Lossless and lossy network)}
The power lines are assumed to be {\em lossless} in \eqref{micro.grid}. This is valid if the lines are dominantly inductive, a condition which can be fulfilled by tuning output impedances of the inverters; see e.g.~\cite{li-kao}.  As will be observed in Subsection \ref{subs.lossy}, the lossless assumption can be relaxed by considering {\em lossy}, yet {\em homogenous}, power lines. 
\end{rem}

\begin{rem} \label{r:loads}{\bf (Loads)}
There are a few load scenarios that can be incorporated in the microgrid model \eqref{micro.grid}. The first scenario accounts for purely inductive loads, see \cite[Remark 1]{schiffer.tcst}. Whether these loads are collocated with inverters or appear as individual nodes, they will lead to nonzero shunt admittances at the nodes of the reduced network, where the latter follows from Kron reduction. The resulting shunt admittances constitute the nonzero shunt susceptance $\hat B_{ii}$ introduced after \eqref{reactive}, see also \cite[Section V.A]{simpson.industrial} and \cite{schiffer.tcst}. 
As for the active power loads, following \cite[Remark 3.2]{schiffer.aut}, one can consider negative active power setpoints $P_i^*$ for the inverter $i$, which corresponds to the inverter $i$ connecting a storage device to the grid, in which case the device is acting as a frequency and voltage dependent load (see also \cite[Section 2.4]{muenz}).  Another possibility is to consider constant active power loads collocated with the inverters by embedding the constant active power consumption in the term $P_i^*$. We remark that the controllers studied in the paper do not rely on the knowledge of $P_i^*$, and are therefore fully compatible with the case in which  $P_i^*$ are not completely known due to uncertainties in the loads. 
Finally, the extension of our analysis to the lossy lines in Subsection \ref{subs.lossy} allows us to accommodate loads as homogenous $RL$ circuits. As an interesting special case of this, the forthcoming dissipativity/stability analysis carries over to the case of microgrids with (purely) resistive lines and loads. 
More details on this case are provided in Subsection \ref{subs.lossy}.
\end{rem}
}

To pursue our analysis, we demonstrate an incremental  {\color{black}cyclo-}dissipativity property of the various microgrid models, with respect to a ``synchronous solution". The notion of dissipativity adopted in this paper is introduced next, and synchronous solutions will be identified afterwards.

\begin{definition}\label{d:dissipativity}
System $\dot x= f(x,u), y=h(x)$,  $x\in \mathcal{X}$, $\mathcal{X}$ the state space, $y,u\in \R^m$, is 
incrementally 
{cyclo-}dissipative with state-dependent supply rate $s(x,u,y)$ and 
with respect to a given input-state-output triple 
$(\overline u,\overline x,\overline y)$, if there exist a continuously differentiable function $\mathcal{S}:{\mathcal{X}}\to\R$, and state-dependent positive {semi}-definite\footnote{A state-dependent matrix $M:\mathcal{X}\to \R^{m\times m}$ is 
		positive semi-definite if $y^T M(x) y\ge 0$ for all $x\in \mathcal{X}$ and for all 
		$y\in \R^m$. If $M$ is positive semi-definite and $y^T M(x) y= 0 \Leftrightarrow y=\boldsymbol{0}$ then $M$ is called positive definite.}
matrices $W,R:\mathcal{X}\to \R^{m\times m}$, such that 
for all $x\in \mathcal{X}$, 
$u \in \R^m$ 
and $y=h(x)$, $\overline y = h(\overline x)$\footnote{{\color{black} We are slightly abusing the classical notion of incremental dissipativity \cite{desoer.vidyasagar}, for we do not consider pairs of arbitrary input-state-output triples, but pairs in which one of the two triples is fixed. For additional work on incremental dissipativity we refer the reader to \cite{pavlov.marconi, stan.sepulchre.tac07}.}}
\[
\frac{\partial \mathcal{S}}{\partial x} f(x,u) +\frac{\partial \mathcal{S}}{\partial \overline x} f(\overline x,\overline u) \le 
s(x,u-\overline u,y-\overline y)
\]
with 
\be\label{sr}
s(x,u,y) = -y^T W(x) y+y^T R(x) u.
\ee
\end{definition}


We remark that at this point the function $\mathcal{S}$ is not required to be non-negative nor bounded from below and that the weight matrices $W,R$ are allowed to be state dependent. The use of the qualifier ``{cyclo}" in the definition above stresses  the former feature \claudio{\cite[Def. 2]{Willems.EJC07}}.
	\begin{remark}
		In case the matrices $W$ and $R$ are state independent, some notable special cases of Definition 
		\ref{d:dissipativity} are obtained as follows:
		\begin{enumerate}[i)]
			\item $W\geq 0$, $R=I$, $\calS \ge 0$ (incremental passivity)
			\item $W>0$, $R=I$, $\calS \ge 0$  
			(output-strict incremental passivity)
			\item $W\geq 0$, $R=I$  (cyclo-incremental passivity)
			\item $W> 0$, $R=I$ (output-strict cyclo-incremental passivity).
		\end{enumerate}
	\end{remark}

%
\smallskip
\noindent {\it Synchronous solution.} 
Given the constant vectors $\overline u_P$ and $\overline u_Q$,
a synchronous solution \claudio{to \eqref{micro.grid}} is defined as the triple
\[
(\theta(t),\omega(t),V(t))=(\overline \theta, \overline \omega, \overline V),
\]
where \claudio{$\overline \theta= \overline \omega{} t + \theta^0$,  $\overline \omega = \mathds{1} \omega^0$, the scalar $\omega^0$ and the vectors $\theta^0$ and $\overline V\in \R^n_{>0}$  are constant. In addition,}
\be\label{feasibility}
\ba{rcl}
\boldsymbol{0} &=& -(\overline \omega-\omega^*) - K_P(\overline  P-P^\ast) +\overline u_P\\ 
\boldsymbol{0}  &=&f(\overline{V},\overline{Q},\overline u_Q)\,,
\ea\ee 
\claudio{where 
\be
\label{barP.barQ}
\ba{rcccl}
\overline P&=&D\Gamma(\overline V) \boldsymbol{\sin}(D^T \overline \theta)&=&D\Gamma(\overline V) \boldsymbol{\sin}(D^T \theta^0),\\[1mm]
\overline Q&=&[\overline V] \mathcal{A}(\boldsymbol{\cos}(D^T \overline \theta)) \overline V&=&[\overline V] \mathcal{A}(\boldsymbol{\cos}(D^T \theta^0)) \overline V.
\ea
\ee}
{\color{black}
Notice that the key feature of a synchronous solution is that the voltage phase angles are rotating with the same frequency, namely 
$\omega^0$, and the differences of these angles are thus constant. Another feature is that the voltage amplitudes are constant.
} 

%
%

\smallskip

\section{Design of {\color{black}Bregman storage} functions
}\label{sec.III}

A crucial step for the Lyapunov based analysis of the coupled nonlinear model \eqref{micro.grid} is constructing a storage function. To this end, we {\color{black}start off with} the following {\color{black} classical}  energy-based function, {\color{black} e.g.~\cite{pai}}
\be\label{e:U-compact}
\ba{rcl}
U(\theta, \omega, V) &=&\displaystyle \frac{1}{2}  \omega^T K_P^{-1}T_P\omega +\frac{1}{2}  \mathds{1}^T Q\\[2mm]
&=& \displaystyle\frac{1}{2}  \omega^T  K_P^{-1}T_P\omega +\frac{1}{2}  V^T \mathcal{A}(\boldsymbol{\cos}(D^T \theta) )  V,
\ea\ee
where we have exploited (\ref{Q.A}) {{\color{black}to write the second equality.} 
Notice that the first term represents the kinetic 
\cter{``energy" (in quotes because the term has the units of power and it does not correspond to the physical inertia),} 
and the second one the sum of the reactive power stored in the links and the power \new{partly associated with the shunt components.}  

}
Next, we compute 
the gradient of the storage function as follows:
\[\ba{rcl}
\displaystyle\frac{\partial U}{\partial \theta} &=& P = D \Gamma(V) \boldsymbol{\sin}(D^T\theta), \\[2mm]
\displaystyle\frac{\partial U}{\partial V} &=& [V]^{-1} Q= [A_0] V - [V]^{-1} |D| \Gamma(V) \boldsymbol{\cos}(D^T\theta)\,.
\ea\]
In the equality above, we are implicitly assuming that each component of the voltage vector never crosses zero. In fact, we shall assume the following:
\begin{assum}\label{forward.invariance}
There exists a subset $\mathcal{X}$ of the state space $\mathbb{T}^n\times \R^n\times \R_{>0}^n$ that is forward invariant along the solutions to (\ref{micro.grid}). 
\end{assum}
Conditions under which this assumption is fulfilled will be provided later in the paper. 

%
%
%
%
Notice that the voltage dynamics identified by $f$ has not yet been taken into account in  the function $U$.  Therefore, to cope with different voltage dynamics (or controllers) we add another component, namely $H(V)$, and define
\be\label{S}
S(\theta, \omega, V)=U(\theta, \omega, V)+H(V).
\ee
We rest our analysis on the following foundational incremental storage function
\be\label{inc.energy.0}\ba{ll}
\mathcal{S}(\theta, \omega, V) =& S(\theta, \omega, V) - S(\overline\theta, \overline\omega,\overline V)
-\barfpartT{S}{\theta} (\theta-\overline\theta)\\
&-\barfpartT{S}{\omega}(\omega-\overline\omega)
-\barfpartT{S}{V}(V-\overline V)
\ea\ee
where we use the conventional notation 
	\[
	\barfpart{F}{x}=\fpart{F}{x}{(\overline x)}, \qquad \barfpartT{F}{x}=(\fpart{F}{x}{(\overline x)})^T
	\]
	for a function $F :\mathcal{X}\to\R$.
{\color{black}The storage function $\calS$, in fact, defines a ``distance" between the value of $S$ at point $(\theta, \omega, V)$ and the value of a first-order Taylor expansion of $S$ around $(\overline\theta, \overline\omega,\overline V)$. 
This construction is referred to as {\em Bregman distance} or {\em Bregman divergence} following \cite{bregman}, and has found its applications in convex programming, clustering, proximal minimization,  online learning, \cter{and proportional-integral stabilisation of nonlinear circuits}; see e.g. \cite{bregman, clustering-bregman, minimization-bregman, online-bregman, jaya}. In thermodynamics, the Bregman distance has its antecedents in the notion of availability function \cite{keenan, alonso.ydstie, willems}}.

The function $\calS$ can be decomposed as 
\be\label{calS}
\calS=\calU+\calH
\ee
where
\[ 
\ba{ll}
\calU (\theta, \omega, V) =& U(\theta, \omega, V) - U(\overline\theta, \overline\omega,\overline V)
-\barfpartT{U}{\theta} (\theta-\overline\theta)\\
&-\barfpartT{U}{\omega}(\omega-\overline\omega)
-\barfpartT{U}{V}(V-\overline V)
\ea
\]
and
\[
\ba{rcl}
\calH(V) &=& H(V) - H(\overline V)
-\barfpartT{H}{V}(V-\overline V).
\ea
\] 
The above identities show that the critical points of $\mathcal{S}$ occur for $\omega=\overline \omega$  and $P=\overline P$ which is a desired property. The critical point of $\calS$ with respect to the $V$ coordinate is determined  by the choice of $H$ which depends on the voltage dynamics. 

%
%
To establish {\color{black} a suitable} incremental dissipativity property {\color{black}of the system with respect to a synchronous solution}, we introduce the output variables
	\be\label{output}
	y=\col ( y_P,\; y_Q)
	\ee
	with
	\[
	y_P=T_P^{-1}\;\fpart{S}{\omega}=K_P^{-1}\omega, \qquad y_Q=T_Q^{-1}\;\fpart{S}{V} ,
	\]
and input variables
\be\label{input}
u=\col ( u_P,\; u_Q).
\ee

In what follows, we differentiate among different voltage controllers and adjust the analysis accordingly by tuning $H$. 

%

\subsection{Conventional droop controller}

The conventional droop controllers are obtained by setting $f$ in \eqref{micro.grid} as 
\be\label{f-droop}
f(V, Q, u_Q)=-V-K_QQ+u_Q
\ee
where $K_Q=[k_Q]$ is a diagonal matrix with positive droop coefficients on its diagonal. 
Note that $u_Q$ is added for the sake of generality and one can set $u_Q=\overline u_Q=K_QQ^\ast+ V^\ast$ for nominal constant vectors $V^\ast$ and $Q^\ast$ to obtain the well known expression of conventional droop controllers, see e.g. \claudio{\cite{chandorkar}}, \cite{zhong}. 
For this choice of $f$, we pick the function $H$ in \eqref{S} as \smod (\cite{schiffer.aut, trip.cdc2014})\emod 
\be\label{H.cdc}
H(V)=\mathds{1}^T K_Q V-{\color{black} (\overline Q+K_Q^{-1}\overline V)}^T \boldsymbol{\ln}(V), 
\ee 
with ${\color{black}\overline Q+K_Q^{-1}\overline V=K_Q^{-1} \overline u_Q\in \R^n_{>0}}$ and $\boldsymbol{\ln}(V)$ defined element-wise.
This term has two interesting features. First, it makes the incremental storage function 
$\mathcal{S}$ radially unbounded with respect to $V$ on the  positive orthant. Moreover, it shifts the critical points of $\calS$ as desired.  
Noting that {\color{black} by \eqref{feasibility}}
\[
0=-\overline V-K_Q\overline Q+\overline u_Q,
\]
{\color{black}straightforward calculations yields}
\be\label{dV--droop}
T_Q\dot{V}=-K_Q[V]\fpart{\calS}{V}+u_Q -\overline u_Q.
\ee
%
%
In the following subsections we will derive analogous identities and then  use those for concluding incremental cyclo-dissipativity of the system.


\subsection{Quadratic droop controller}

Another voltage dynamics proposed in the literature
is associated with the quadratic droop controllers of \cite{simpson.quadratic},  which can be expressed as \eqref{micro.grid} with
\be\label{f-Qdroop}
f(V, Q, u_Q)=-K_QQ-[V] (V-u_Q), 
\ee
where again $K_Q=[k_Q]$ collects the droop coefficients. The quadratic droop controllers in \cite{simpson.quadratic} is obtained by setting $u_Q=V^\ast$ for some constant vector $V^\ast$. \smod Notice however the difference: while  \cite{simpson.quadratic} focuses on a network preserved microgrid model in which the equation above models the inverter dynamics and are decoupled from the frequency dynamics, here a fully coupled network reduced model is considered. \emod

\emod
Moreover, note that the scaling matrix $[V]$ distinguishes this case from the conventional droop controller. For this case, we {adapt} the storage function $\calS$ by setting
\be\label{H.qdc}
H(V)=\frac{1}{2}V^TK_Q^{-1}V.
\ee
Recall that $\calS=\calU+\mathcal{H}.$ 
Note that $\calS$ is defined on the whole $\mathbb{T}^n\times \R^n\times \R^n$ and not on $\mathbb{T}^n\times \R^n\times \R^n_{>0}$. The resulting function $\calS$  can be interpreted as a performance criterion in a similar vein as the cost function in \cite{simpson.quadratic}. 
Noting that
\[
0=-K_Q\overline Q-[\overline V](\overline V-\overline u_Q),
\]
{\color{black}it is easy to verify that}
\be\label{dV--Qdroop}
T_Q\dot{V}=\new{-}K_Q[V]\fpart{\calS}{V}+[V](u_Q-\overline u_Q).
\ee

\subsection{Reactive current controller}
The frequency dynamics of the inverters in microgrids typically mimics  that of the synchronous generators known as the {\em swing equation}. 
This facilitates the interface of inverters and generators in the grid.
To enhance such interface, an idea is to mimic the voltage dynamics of the synchronous generators as well. Motivated by this, we consider the voltage dynamics identified by
\be\label{f-current}
f(V, Q, u_Q)=-[V]^{-1}Q+u_Q.
\ee
This controller aims at regulating the ratio of reactive power over voltage amplitudes, which can be interpreted as ``reactive current" 
 (\cite{macho}). \emodc 
For this controller, we set 
\be\label{H.rcc}
H=0
\ee 
meaning that $S=U$ and no adaption of the storage function is needed. 
It is easy to observe that
\be\label{dV--current}
T_Q\dot{V}=-\frac{\partial \calS}{\partial V}+u_Q-\bar u_Q,
\ee
where $\bar u_Q=[\overline V]^{-1}\overline Q$ is again the feedforward input guaranteeing the preservation of the steady state.

\subsection{Exponentially-scaled averaging reactive power controller}
In this subsection, we consider  another controller which aims at achieving proportional {reactive} power  sharing 
\be\label{f-consensus}
f(V, Q, u_Q)=-[V]K_Q\new{L_Q} K_QQ+[V]u_Q
\ee
where $K_Q=[k_Q]$ is a diagonal matrix and $L_Q$ is the Laplacian matrix of a communication graph which is assumed to be undirected and connected. Compared with the controller in \cite{schiffer.tcst}, here the the voltage dynamics is scaled by the voltages at the inverters, namely $[V]$, {the reactive power $Q$ is {\em not} assumed to be independent of the phase variables $\theta$,} and an additional input $u_Q$ is introduced.
It is easy to see that the voltage dynamics in this case can be equivalently rewritten as
\be\label{e-arp}
\ba{rcl}
\claudio{T_Q} \dot{\chi}&=&-K_QL_QK_Q Q+u_Q\\[2mm]
V&=&\boldsymbol{\mathrm{exp}}{(\chi)}
\ea
\ee
where {\color{black}$Q$ can be expressed in terms of $\chi$ as $[\boldsymbol{\mathrm{exp}}{(\chi)}] \mathcal{A}(\boldsymbol{\cos}(D^T \theta)) \boldsymbol{\mathrm{exp}}{(\chi)}$ with $\boldsymbol{\mathrm{exp}}(\chi)=\col(e^{\chi_i})$.} Hence, we refer to this controller  as an   {\em exponentially-scaled averaging reactive power} controller (E-ARP). Now, we
choose $H$ as
\be\label{H.cbc}
H(V)=-\overline Q^T\boldsymbol{\ln}V,
\ee
{\color{black} with $\overline Q$ as in \eqref{barP.barQ},  and obtain}
\be\label{dSdV--consensus}
\fpart{\calS}{V}=[V]^{-1}(Q-\overline Q). 
\ee
Note that, in fact, our treatment here together with the above equality hints at the inclusion of the matrix $[V]$ 
into the controller, 
or equivalently at an {\em exponential scaling} of the reactive power averaging dynamics  
(see \eqref{f-consensus}, \eqref{e-arp}).  This, as will be observed, results in reactive power sharing for the fully coupled nonlinear model \eqref{micro.grid}.
By defining  
\be\label{baruQ-consensus}
\overline u_Q=K_QL_Q K_Q\overline Q,
\ee
the {\color{black} voltage} dynamics 
can be rewritten as  
\be\label{inview}
\dot V = -[V]K_QL_Q K_Q[V] \fpart{\calS}{V}+[V](u_Q- \overline u_Q). 
\ee
\claudio{where we have set $\claudio{T_Q}=I$. Having unitary time constants is assumed for the sake of simplicity and could be relaxed. On the other hand, requiring them to be the same is a purely technical assumption, motivated by the difficulty of analysing the system without such condition.}

\subsection{Incremental dissipativity
of microgrid models
}

In this subsection, we show how the candidate \claudio{Bregman} storage functions introduced before allow us to infer incremental dissipativity of the microgrids under the various controllers.

\begin{thm}\label{t:droop}
Assume that the feasibility condition \eqref{feasibility} admits a solution and let Assumption \ref{forward.invariance} hold. 
Then system \eqref{micro.grid} with  output \eqref{output}, input \eqref{input}, and, respectively,  
\begin{enumerate}
\medskip{}
\item $f(V, Q, u_Q)$ given by \eqref{f-droop};
\medskip{}
\item $f(V, Q, u_Q)$ given by \eqref{f-Qdroop}; 
\medskip{}
\item $f(V, Q, u_Q)$ given by \eqref{f-current}; 
\medskip{}
\item $f(V, Q, u_Q)$ given by \eqref{f-consensus}; 
\medskip{}
\end{enumerate}
is incrementally \new{cyclo-}dissipative with respect to a synchronous solution $(\overline \theta, \overline \omega, \overline V)$,  with 
\begin{enumerate}
\item incremental storage function $\calS$ defined by 
\eqref{e:U-compact},\eqref{S},\eqref{inc.energy.0},\eqref{H.cdc} and supply rate (\ref{sr}) with  
weight matrices 
\[
W(V) = \begin{bmatrix}
 K_P  & \boldsymbol{0}\\
 \boldsymbol{0} & T_Q K_Q[V] 
\end{bmatrix}, \quad 
R= \begin{bmatrix}
I  & \boldsymbol{0}\\
\boldsymbol{0} & I
\end{bmatrix};
%
\]
\item incremental storage function $\calS$ defined by 
\eqref{e:U-compact},\eqref{S},\eqref{inc.energy.0},\eqref{H.qdc} and supply rate (\ref{sr})  with  
weight matrices 
\[\ba{c}
W(V) = \begin{bmatrix}
 K_P  & \boldsymbol{0}\\
 \boldsymbol{0} & T_Q K_Q[V] 
\end{bmatrix}, \\[5mm]
R(V)= \begin{bmatrix}
 I  & \boldsymbol{0}\\
 \boldsymbol{0} & [V] 
\end{bmatrix};
\ea\]
\item incremental storage function $\calS$ defined by 
\eqref{e:U-compact},\eqref{S},\eqref{inc.energy.0},\eqref{H.rcc}  and supply rate (\ref{sr})  with  
weight matrices 
\[
W = \begin{bmatrix}
 K_P & \boldsymbol{0}\\
 \boldsymbol{0} & T_Q
\end{bmatrix}, \quad 
R= \begin{bmatrix}
 I  & \boldsymbol{0}\\
 \boldsymbol{0} & I
\end{bmatrix};
\]
\item incremental storage function $\calS$ defined by 
\eqref{e:U-compact},\eqref{S},\eqref{inc.energy.0},\eqref{H.cbc}  and supply rate (\ref{sr})  with  
weight matrices 
\[\ba{c}
W(V) = \begin{bmatrix}
 K_P  & \boldsymbol{0}\\
 \boldsymbol{0} &   [V]K_Q L_Q K_Q[V]  
\end{bmatrix}, \\[5mm]
R(V)= \begin{bmatrix}
I  & \boldsymbol{0}\\
 \boldsymbol{0} &  [V]  
\end{bmatrix}.
\ea
\] 
\end{enumerate}
%
%
\end{thm}

\begin{proof}
1) Recall that
\be\label{recall}\ba{c}
\displaystyle\frac{\partial \calS}{\partial  \omega}=K_P^{-1}T_P( \omega-\overline \omega), \\[3mm] 
\displaystyle\frac{\partial \calS}{\partial \theta}= D\Gamma(V) \boldsymbol{\sin}(D^T\theta)- D\Gamma(\overline V) \boldsymbol{\sin}(D^T \theta^0) =(P - \overline P).
%
\ea\ee
Then 

\[
\ba{l}
\displaystyle\frac{d}{dt}\calS=( \omega-\overline \omega)^T T_PK_P^{-1} \dot \omega\\[2mm]
+
(D\Gamma(V) \boldsymbol{\sin}(D^T(\theta))
\;-D\Gamma(\overline V) \boldsymbol{\sin}(D^T \theta^0))^T \dot \theta
+(\displaystyle\frac{\partial \calS}{\partial V})^T\dot V
\\[2mm]
=
( \omega-\overline \omega)^TK_P^{-1} (-(\omega -\overline \omega) - K_P (P-\overline P)
+ (u_P-\overline u_P))  \\[2mm]
\;+(D\Gamma(V) \boldsymbol{\sin}(D^T \theta)- D\Gamma(\overline V) \boldsymbol{\sin}(D^T \theta^0))^T( \omega-\overline \omega)\\[2mm]
\;+(\displaystyle\frac{\partial \calS}{\partial V})^T T_Q^{-1}(- K_Q[V]  \displaystyle\frac{\partial \calS}{\partial V}+u_Q-\bar u_Q)
\\[2mm]
=
( \omega-\overline \omega)^TK_P^{-1}(-(\omega -\overline \omega) - {\color{black}K_P (P-\overline P)}+ (u_P-\overline u_P))\\[2mm]
\;+{\color{black}(P-\overline P)^T 
( \omega-\overline \omega)}-(\displaystyle\frac{\partial \calS}{\partial V})^T T_Q^{-1}K_Q[V]  \displaystyle\frac{\partial \calS}{\partial V}\\[2mm]
\;+ (\displaystyle\frac{\partial \calS}{\partial V})^T T_Q^{-1} (u_Q-\bar u_Q)
%
%
\ea
\]
where the chain of equalities hold because of the feasibility condition and (\ref{dV--droop}). 
Hence
\be\label{dS-droop}
\ba{l}
\displaystyle \frac{d}{dt} \calS=
-( \omega-\overline \omega)^T K_P^{-1}(\omega -\overline \omega) + ( \omega-\overline \omega)^TK_P^{-1} (u_P-\overline u_P)
\\
\quad \qquad-(\displaystyle\frac{\partial \calS}{\partial V})^T T_Q^{-1} K_Q[V]  \displaystyle\frac{\partial \calS}{\partial V}+ (\displaystyle\frac{\partial \calS}{\partial V})^T T_Q^{-1} (u_Q-\bar u_Q).
\ea
\ee
Observe now that by definition 
\[
\displaystyle\frac{\partial \calS}{\partial V} =
\displaystyle \frac{\partial S}{\partial V}-\displaystyle \left.\frac{\partial S}{\partial V}\right|_- 
\]
and that $\left.\frac{\partial S}{\partial V}\right|_-$ represents the  output component $ \frac{\partial S}{\partial V}$ at a synchronous solution. Hence  equality (\ref{aux.lab}) at the top of the next page can be established. 

\begin{figure*}[!t]
\normalsize
\setcounter{mytempeqncnt}{\value{equation}} 
\setcounter{equation}{30}
{
\be\label{aux.lab}\ba{rcl}
\displaystyle\frac{d}{dt}\calS &=& 
-
\begin{bmatrix}
(\omega -\overline \omega)^TK_P^{-1} & (\frac{\partial S}{\partial V}- \left.\frac{\partial S}{\partial V}\right|_-)^T T_Q^{-1} 
\end{bmatrix}
\begin{bmatrix}
 K_P & \boldsymbol{0}\\
 \boldsymbol{0} & T_QK_Q[V] 
\end{bmatrix}
\begin{bmatrix}
K_P^{-1} (\omega -\overline \omega) \\ T_Q^{-1}  (\frac{\partial S}{\partial V}- \left.\frac{\partial S}{\partial V}\right|_-)
\end{bmatrix}\\
&& +\begin{bmatrix}
(\omega -\overline \omega)^TK_P^{-1} & (\frac{\partial S}{\partial V}- \left.\frac{\partial S}{\partial V}\right|_-)^TT_Q^{-1}
\end{bmatrix}
\begin{bmatrix}
 I  & \boldsymbol{0}\\
 \boldsymbol{0} & I
\end{bmatrix}
\begin{bmatrix}
u_P-\bar u_P \\ u_Q-\bar u_Q
\end{bmatrix}.
\ea\ee
}
\setcounter{equation}{\value{mytempeqncnt}+1}
\hrulefill
\vspace*{4pt}
\end{figure*}

\noindent We conclude  {incremental cyclo-dissipativity} of system (\ref{micro.grid}), \eqref{output},  (\ref{input}), \eqref{f-droop} as claimed. 
\\
2)  If in the chain of equalities defining $\displaystyle\frac{d}{dt}\calS$ above, we use (\ref{dV--Qdroop}) instead of (\ref{dV--droop}),   we obtain that 
\be\label{dU-Qdroop}
\ba{l}
\displaystyle\frac{d}{dt}{\calS}= -( \omega-\overline \omega)^T K_P^{-1}(\omega -\overline \omega) + ( \omega-\overline \omega)^TK_P^{-1}(u_P-\overline u_P)\\
\quad \qquad -(\displaystyle\frac{\partial {\calS}}{\partial V})^T T_Q^{-1} K_Q[V]  \displaystyle\frac{\partial  {\calS}}{\partial V}+ (\displaystyle\frac{\partial \calS}{\partial V})^T T_Q^{-1}[V](u_Q-\overline u_Q)
\ea\ee
which shows \alert{incremental cyclo-dissipativity}  of  system (\ref{micro.grid}), \eqref{output},  (\ref{input}), \eqref{f-Qdroop}. 
\\
3) For this case, adopting the equality (\ref{dV--current}) results in the equality 
\be\label{dU-current}
\ba{l}
\displaystyle\frac{d}{dt}{\calS} = -( \omega-\overline \omega)^TK_P^{-1} (\omega -\overline \omega) + ( \omega-\overline \omega)^TK_P^{-1}(u_P-\overline u_P)\\
\qquad \qquad 
-\displaystyle (\frac{\partial \calS}{\partial V})^TT_Q^{-1}\frac{\partial \calS}{\partial V}+ \frac{\partial \calS}{\partial V}^TT_Q^{-1}(u_Q-\overline u_Q),
\ea
\ee
from which \alert{incremental cyclo-dissipativity}  of \eqref{micro.grid}, \eqref{output}, \eqref{input}, (\ref{f-current})  holds.  \\
4) Finally, in view of (\ref{inview}), 
\be\label{dU-consensus}
\ba{l}
\displaystyle\frac{d}{dt}{\calS}= -( \omega-\overline \omega)^TK_P^{-1}(\omega -\overline \omega) + ( \omega-\overline \omega)^TK_P^{-1}(u_P-\overline u_P)
\\
\quad -(\displaystyle\frac{\partial {\calS}}{\partial V})^T [V]K_Q L_Q K_Q[V]\displaystyle\frac{\partial  {\calS}}{\partial V}
+(\displaystyle\frac{\partial \calS}{\partial V})^T  [V](u_Q-\overline u_Q)\ea\ee
which implies \alert{incremental cyclo-dissipativity}  of (\ref{micro.grid}), \eqref{output}, \eqref{input}, (\ref{f-consensus}). 
\end{proof}
\section{From cyclo-dissipativity to dissipativity}\label{sec.IV}
The dissipation inequalities proven before can be exploited to study the stability of a synchronous solution. 
\nima{Recall that} Theorem \ref{t:droop} has been established  in terms of 
cyclo-dissipativity rather than dissipativity, i.e. without imposing lower boundedness of the storage function $\calS$. However, in order to conclude the attractivity of a synchronous solution we ask for incremental dissipativity of the system, and require the storage function to possess  a strict minimum at the point of interest.
To this end, we investigate conditions under which the Hessian of the storage function $\calS$ is positive definite at the point of interest identified by a synchronous solution.

 It is not difficult to observe that due to the rotational invariance of $\theta$ variable, the existence of a {\em strict} minimum for $\calS$ cannot be anticipated. To clear  this obstacle, we notice that the phase angles $\theta$ appear as relative terms, i.e. 
 {\color{black}$D^T \theta$,}
 in \eqref{e:U-compact} and thus in $S$ as well as $\calS$.
 \claudio{
Motivated by this observation, we introduce the new variables {\color{black}\cite{jl.willems}}
\be\label{varphi}
\varphi_i = \theta_i -\theta_n, \quad i=1,2,\ldots, n-1.
\ee
These can be also written as
\[
\begin{bmatrix}
\varphi_1 \\ \vdots \\ \varphi_{n-1} \\ 0
\end{bmatrix}=
\begin{bmatrix}
\theta_1 \\ \vdots \\ \theta_{n-1} \\ \theta_n
\end{bmatrix}-\mathds{1} \theta_n.
\]
Let us partition $D$ accordingly as $D={\rm col} (D_1, D_2)$, \cter{with $D_1$ an $(n-1)\times m$ matrix and $D_2$ a  $1\times m$ matrix. Notice that $D_1$ is the reduced incidence matrix corresponding to the node of index $n$ taken as reference}. Then
\[
D^T \begin{bmatrix}
\varphi_1 \\ \vdots \\ \varphi_{n-1} \\ 0
\end{bmatrix}= D_1^T \varphi, \textrm{with } \varphi:=\begin{bmatrix}
\varphi_1 \\ \vdots \\ \varphi_{n-1} 
\end{bmatrix}
\]
and therefore
\[
D^T \theta= D_1^T \varphi. 
\]
More explicitly, given $\theta \in \R^n$, we can define $\varphi\in \R^{n-1}$ from $\theta$ as in (\ref{varphi}), and the equality $D^T \theta= D_1^T \varphi$ holds. 
Hence, 
{\color{black}
\[\ba{rcl}
U(\theta, \omega, V)&=&\displaystyle\frac{1}{2}  \omega^T K_P^{-1}T_P\omega +\frac{1}{2}  V^T \mathcal{A}(\boldsymbol{\cos}(D^T \theta) )  V\\[2mm]
&=&
\displaystyle\frac{1}{2}  \omega^T K_P^{-1}T_P\omega +\frac{1}{2}  V^T \mathcal{A}(\boldsymbol{\cos}(D_1^T\varphi) )  V
\ea\]
and we set, by an abuse of notation, 
\[
U(\varphi, \omega, V):=\frac{1}{2}  \omega^T K_P^{-1}T_P\omega +\frac{1}{2}  V^T \mathcal{A}(\boldsymbol{\cos}(D_1^T\varphi) )  V.
\]
}
Furthermore, we can define
\be\label{calU-eta}
\ba{c}
\calU (\varphi, \omega, V)= U (\varphi, \omega, V) - U(\overline \varphi, \overline\omega, \overline V)
\\[2mm]
\quad -\displaystyle \left.\frac{\partial U}{\partial \varphi}\right|_-^T
(\varphi-\overline\varphi)
-\displaystyle \left.\frac{\partial U}{\partial \omega}\right|_-^T (\omega-\overline\omega)
-\displaystyle \left.\frac{\partial U}{\partial V}\right|_-^T (V-\overline V)
\ea
\ee
where, $\overline\varphi_i := \overline\theta_i -\overline\theta_n$, $i=1,2,\ldots, n-1$,  (hence $D^T \overline \theta=D_1^T \overline \varphi$), and
\be\label{cal.S}
\calS(\varphi, \omega, V)=\calU(\varphi, \omega, V)+\calH(V)
\ee
to have
\be\label{shift-co}
\calS(\theta, \omega, V)=\calS(\varphi, \omega, V).
\ee
}

\subsection{\claudio{Strict convexity of Bregman storage functions}}
\claudio{Observe that} $(\overline \varphi, \overline \omega, \overline V)$ is a critical point of $\calS$.
Next, we compute the Hessian as
\be\label{Hessian-3*3}
\ba{c}
\displaystyle\frac{\partial^2 S}{\partial (\varphi, \omega, V)^2}=\\[4mm]
{\footnotesize
\begin{bmatrix}
D_1 \Gamma(V) [\boldsymbol{\cos}(D_1^T\varphi)] D_1^T & \boldsymbol{0}& *\\[2mm]
\boldsymbol{0} & K_P^{-1} T_P & \boldsymbol{0}\\[2mm]
D_1 [V]^{-1} |D| \Gamma(V) [\boldsymbol{\sin}(D_1^T\varphi)] & \boldsymbol{0}  & \mathcal{A}(\boldsymbol{\cos}(D_1^T\varphi))+
\displaystyle\frac{\partial^2 \cter{H}}{\partial V^2}
\\[2mm]
\end{bmatrix}
}
.
\ea
\ee
%

\smallskip
\noindent Notice that in all the previously studied cases, the matrix $\frac{\partial^2 H}{\partial V^2}$ is diagonal. 
In particular, 
\be\label{h1}
\ba{ll}
\displaystyle\frac{\partial^2 H}{\partial V^2}=K_Q+[V]^{-2}[{\color{black}\overline Q+K_Q^{-1}\overline V}],&
\displaystyle\frac{\partial^2 H}{\partial V^2}=K_Q^{-1},\\[3mm]
\displaystyle\frac{\partial^2 H}{\partial V^2}=0,&
\displaystyle\frac{\partial^2 H}{\partial V^2}=[V]^{-2}[{\color{black} \overline Q}],
\ea
\ee
for conventional droop, quadratic droop, reactive current controller, and exponentially-scaled averaging reactive power controller, respectively. Now, let 
\be\label{h2}
[h(V)]:=\displaystyle\frac{\partial^2 H}{\partial V^2},
\ee
and $h(V)=\col(h_i(V_i))$. Then, the following result, which establishes {\color{black}distributed} conditions for checking the positive definiteness of the Hessian, \claudio{and hence strict convexity of the Bregman storage function,} can be proven:

\begin{proposition}\label{p1}
Let $\overline \eta:=D^T\theta^0=D_1^T \overline \varphi \in (-\frac{\pi}{2}, \frac{\pi}{2})^m$, $\overline  V\in \mathbb{R}^n_{>0}$, and
\[
m_{ii} :=
\hat B_{ii} + \displaystyle\sum_{{k\sim \{i, \ell\}\in E}} B_{i \ell} 
\left(
1-\frac{\overline V_\ell}{\overline V_i} \frac{\sin^2 ({
\overline \eta_k
})}{\cos({
\overline \eta_k
})} 
\right)
+ 
h_i(\overline V_i).
\]
Then the inequality 
\be\label{sconvex}
\left.\frac{\partial^2 S}{\partial (\varphi, \omega, V)^2}\right|_->0
\ee
holds if
\claudio{
\be\label{cond2}
m_{ii}>  
\displaystyle
\sum_{
{k\sim\{i, \ell\}\in E}} B_{i \ell} \;
{
\sec(\overline \eta_k)
}
\ee
}
for all $i=1,2,\ldots, n$.
\end{proposition}

\begin{proof}
The proof is given in the appendix. 
\end{proof}

\begin{rem}
\claudio{{\bf (Isolated minima)}}
The result shows that 
the condition \eqref{cond2} for positive definiteness are met provided that {at the point} $({\color{black}\overline \varphi}, \overline \omega, \overline V)$ the relative voltage phase angles are small enough and the voltages magnitudes are approximately the same. This is a remarkable property, stating that {if the equilibria of interest are characterized by small relative voltage phases and similar voltage magnitudes, then they  are  minima of the incremental storage function $\calS(\theta, \omega, V)$, and equivalently {\em isolated}} minima of  $\calS(\varphi, \omega, V)$.  \qedp
\end{rem}
{
\begin{rem} \cter{{\bf (Hessian)}}\label{r:Hess}
The Hessian of energy functions has always played an important role in stability studies of power networks (see e.g.~\cite{tsolas}, and \cite{schiffer.aut} for a microgrid stability investigation).  
Conditions for assessing the positive definiteness of the Hessian of an energy function associated to power networks have been  reported in the literature since \cite{tsolas}, and  used even recently to study e.g.~the convexity of the energy function (\cite{dvijo}). Our conditions however 
are different and hold for more general energy functions. 
\qedp \end{rem}}

\subsection{An instability condition}
Conversely, one can characterize an instability condition that shows how, for a given vector of voltage values,  equilibria with ``large" relative phase angles are unstable. 
To this end, first observe that {\color{black} a negative eigenvalue of the Hessian matrix} implies instability of the equilibrium $(\overline \varphi, \overline \omega, \overline V)$ \claudio{of system \eqref{micro.grid}, with $f(V,Q,u_Q)$ given by \eqref{f-droop}, \eqref{f-Qdroop}, \eqref{f-current}, expressed in the $\varphi$ coordinates and with $u_P=\overline u_P$, $u_Q=\overline u_Q$:}
\begin{lem}\label{inertia}
Suppose that the Hessian
\[
\left.\frac{\partial^2 S}{\partial (\varphi, \omega, V)^2}\right|_-
\]
has a negative eigenvalue. Then the equilibrium $(\overline \varphi, \overline \omega, \overline V)$ is unstable. 
\end{lem}
\begin{proof}
The proof is given in the appendix. 
\end{proof}

%

\cter{Before providing sufficient conditions under which the Hessian in Lemma \ref{inertia} has a negative eigenvalue, we first provide conditions under which  the matrix  at the center of the product in (\ref{Hess-1}), denoted as $M$ when evaluated at $(\overline \varphi, \overline \omega, \overline V)$, has a negative eigenvalue.  $M$ 
is symmetric and as such diagonalizable.  Using the diagonal form, it is immediate to notice that if there exists a vector $v=(v^{(1)}, v^{(2)})\ne 0$ such that $v^T M v <0$, 
then the matrix $M$ has a negative eigenvalue. }
%
%
%
%
%
\begin{figure*}[!t]
\normalsize
\setcounter{mytempeqncnt}{\value{equation}} 
\setcounter{equation}{43}
{
\be\label{Hess-1}
\ba{ll}
&
\begin{bmatrix}
D_1 \Gamma(V) [\boldsymbol{\cos}(D_1^T\varphi)] D_1^T &[\boldsymbol{\sin}(D_1^T\varphi)]  \Gamma(V)  |D|^T  [V]^{-1} D_1^T\\[2mm]
D_1 [V]^{-1} |D| \Gamma(V) [\boldsymbol{\sin}(D_1^T\varphi)] & \mathcal{A}(\boldsymbol{\cos}(D_1^T\varphi))+\displaystyle\frac{\partial^2 H}{\partial V^2}
\\[2mm]
\end{bmatrix}
\\[8mm]
=
\begin{bmatrix}
D_1 & \boldsymbol{0}\\
\boldsymbol{0} & I
\end{bmatrix}
&
\begin{bmatrix}
\Gamma(V) [\boldsymbol{\cos}(D_1^T\varphi)] &[\boldsymbol{\sin}(D_1^T\varphi)]  \Gamma(V)  |D|^T  [V]^{-1} \\[2mm]
[V]^{-1} |D| \Gamma(V) [\boldsymbol{\sin}(D_1^T\varphi)] & \mathcal{A}(\boldsymbol{\cos}(D_1^T\varphi))+\displaystyle\frac{\partial^2 H}{\partial V^2}
\\[2mm]
\end{bmatrix}
\begin{bmatrix}
D_1^T & \boldsymbol{0}\\
\boldsymbol{0} & I
\end{bmatrix}
>0.
\ea
\ee
}
\setcounter{equation}{\value{mytempeqncnt}+1}
\hrulefill
\vspace*{4pt}
\end{figure*}

A characterization of the condition $v^T M v <0$, or equivalently the existence of a {\color{black} negative} eigenvalue of the matrix $M$,  is now studied. To this end, it is instrumental to introduce a class of cut-sets of the graph, \cter{as in the following definition:} 
\cter{
\begin{definition}
A cut-set $K\subset E$ is said to have non-incident edges if for each $k\sim \{i,j\}\in K$ and $k\sim \{i',j'\}\in K$, with $k\ne k'$, all the indices $i, j, i', j'$ are different from each other. 
The class of cut-sets with non-incident edges is denoted by $\mathcal{K}$. 
\end{definition}
}
In words, the property amounts to the following: given any two edges in the cut\cter{-set}, the two pairs of end-points associated with the two edges are different from each other. The set of graphs for which these cuts exists is not empty and includes trees, rings and lattices. Complete graphs do not admit this class of cuts. 
\\
The following holds: 
\begin{lem}\label{lemma.cut}
\cter{Let $\overline  V\in \mathbb{R}^n_{>0}$.}
If there exists a cut-set $K$ in the class $\mathcal{K}$ such that, for all $k\sim\{i,j\}\in K$,  
\be\label{instab.cond}
\sin(\overline \eta_k)^2 >  \beta_k(\overline V_i, \overline V_j) \cos(\overline \eta_k),
\ee
where \claudio{$\overline \eta = D_1^T \overline \varphi$ and }
\[\beta_k(\overline V_i, \overline V_j)=\textstyle 2\max\{\frac{(B_{ii}+h_i(\overline  V_i))  \overline V_i}{B_{ij} \overline V_j}, \frac{(B_{jj}+h_j(\overline  V_j))  \overline V_j}{B_{ij} \overline V_i} \},
\]
%
%
and $h_i$ is defined in \eqref{h1}, \eqref{h2}, then the matrix $M$ \cter{at the center of the product} in \eqref{Hess-1} evaluated at $\overline \varphi, \overline V$ has a {\color{black}negative} eigenvalue.  
\end{lem}

\begin{proof}
The proof is given in the appendix. 
\end{proof}


The two lemma above lead to the following conclusion:
\begin{proposition}\label{prop.instability}
An equilibrium $(\overline \varphi, \overline \omega, \overline V)$, \cter{with 
$\overline  V\in \mathbb{R}^n_{>0}$,}
is unstable if
there exists a  cut-set $K$ in the class $\mathcal{K}$ such that 
the inequality (\ref{instab.cond}) holds for all $k\sim\{i,j\}\in K$.
\end{proposition}
\begin{proof}
The proof is given in the appendix. 
\end{proof}

 From the relation above, we see that for equilibria for which the components of $\overline V$ have comparable values, inequality  
(\ref{instab.cond}) is likely to be fulfilled as $\eta_k$ diverges from $0$, thus showing that equilibria with large relative phase angles are likely to be unstable.   

\begin{remark}
{\bf (Plastic coupling strength)}
It is interesting to establish a connection with existing studies on oscillator synchronization arising in different contexts. Once again, this connection leverages the use of the energy function. If the coupling between any pair of nodes $i,j$ is represented by a single variable $v_{ij}$, modeling e.g.~a dynamic coupling,  instead of the product of the voltage variables $V_i V_j$, then a different model arises. To obtain this, we focus for the sake of simplicity on oscillators without inertia, and replace the previous energy  function (\ref{e:U-compact}) with
\[
U(\theta,v) = -\frac{1}{2}\sum_{i=1}^n \sum_{j\in \mathcal{N}_i} v_{ij} B_{ij} \cos(\theta_j-\theta_i)+\frac{1}{2}\sum_{\{i,j\}\in E} v_{ij}^2.
\]
Then
\[
\frac{\partial U}{\partial v_{ij}} = -B_{ij} \cos(\theta_j-\theta_i) + v_{ij},
\]
and the resulting (gradient) system becomes 
\[\ba{rcl}
\dot \theta_i &=& \sum_{j\in \mathcal{N}_i} v_{ij} B_{ij}   \sin(\theta_j-\theta_i),\quad i=1,2,\ldots,n\\
\dot v_{ij} &=& \cter{B_{ij} \cos(\theta_j-\theta_i) - v_{ij}},\quad \{i,j\}\in E, 
\ea\]
which arises in oscillator networks with so-called plastic coupling strength (\cite{scardovi.cdc10,gushchin.et.al.ita15,mallada.physics}) and in the context of flocking with state dependent sensing
(\cite{scardovi.cdc10, krishna, sepulchre}). Although 
stability
analysis of equilibria have been carried out for these systems,  the investigation  of the methods proposed in this paper in those contexts is still unexplored and deserves attention. 
\end{remark}

\vspace*{-0.5cm} 
\section{\claudio{Frequency control with  power sharing}}\label{sec.attractivity}
%
%

In this section, we establish the attractivity of a synchronous solution, which amounts to the frequency regulation ($\overline \omega=\omega^*$) with  optimal properties. Moreover, we investigate voltage stability and reactive power sharing in the aforementioned voltage controllers. 

Recall {from \eqref{feasibility}} that for a synchronous solution we have  
\be\label{syncw-ss}
0=-K_P(D\Gamma(\overline V)\boldsymbol{\sin}(D^T\theta^0)-P^\ast)+\overline u_P.
\ee
Among all possible vectors $\overline u_P$ satisfying the above, we look for the one that minimizes the quadratic cost function 
\be\label{e:quadratic}
C(\overline u_P)=\frac{1}{2}\overline u_P^TK_P^{-1} \overline u_P.
\ee
This choice is explicitly computed as {\color{black}\cite{andreasson.acc14, dorfler.et.al.arxiv14, ST:MB:CDP:AUT15}}
\be\label{optimal.u_P}
\overline u_P =- \mathds{1}\frac{\mathds{1}^TP^\ast}{\mathds{1}^T K_P^{-1} \mathds{1}}\,.
\ee
\claudio{Note that in \eqref{e:quadratic} 
any positive diagonal matrix, say $\Sigma$,  could be used instead of $K_P^{-1}$. However, the choice $\Sigma=K_P^{-1}$ yields more compact expressions, and 
results in proportional sharing of the active power according the droop coefficients $k_{P,i}$, see Subsection \ref{subset.power.sharing}.}  

{\color{black} Replacing $\overline u_P$ in \eqref{feasibility} with its expression in  \eqref{optimal.u_P}, and replacing $\overline Q$ with its explicit definition via the loopy Laplacian, the feasibility condition \eqref{feasibility} can be restated as follows:} 

\medskip{}
\begin{assum}\label{assum.feasibility}
There exist constant vectors $\overline V\in \R^n$ and $\theta^0\in \mathbb{T}^n$ such that 
\be\label{feas-opt-1}
D\Gamma(\overline V) \boldsymbol{\sin}(D^T \theta^0) = 
(I- K_P^{-1}  \frac{\mathds{1}\mathds{1}^T}{\mathds{1}^T K_P^{-1} \mathds{1}}) P^*
\ee
and 
\be\label{feas-opt-2}
0=f(\overline V, [\overline V] \mathcal{A}(\boldsymbol{\cos}(D^T \theta^0)) \overline V, \overline u_Q).
\ee
\end{assum}
\medskip{}

\begin{remark}
Similar to \cite[Remark 5]{ST:MB:CDP:AUT15} it can be shown that if the assumption above is satisfied then necessarily $\overline V\in \mathbb{R}^n_{>0}$. Furthermore, in case the network is a tree, it is easy to observe that \eqref{feas-opt-1} is satisfied if 
and only if there exists $\overline V\in \mathbb{R}^n_{>0}$ such that
\[
\|\Gamma(\overline V)^{-1} D^\dag (I- K_P^{-1}  \frac{\mathds{1}\mathds{1}^T}{\mathds{1}^T K_P^{-1} \mathds{1}}) P^*\|_\infty <1,
\]
with $D^\dag$ denoting the left inverse of $D$. {In the case of the quadratic voltage droop and \cter{reactive current} controllers, explicit expressions of the voltage vector $\overline V$ can be given (see Subsection \ref{subset.power.sharing}), in which case the condition above becomes dependent on the voltage phase vector 
$\theta^0$ only. }   
\end{remark}
 
To achieve the optimal input \eqref{optimal.u_P}, we 
consider the following 
\claudio{distributed} active power controller (\cite{simpson, dorfler.et.al.arxiv14, mb:cdp:st:mtns14}) 
{
\be\label{active.controller}
\ba{rcl}
\dot \xi &=& -L_P \xi+ K_P^{-1}(\omega^*-\omega)\\
u_P &=&  \xi
\ea
\ee}
where the matrix $L_P$ is the Laplacian matrix of an undirected and connected communication graph. 
\claudio{Here, the term $\omega^\ast-\omega$ attempts to regulate the frequency to the nominal one whereas the consensus based algorithm $-L_P\xi$ steers the input to the optimal one given by  \eqref{optimal.u_P} at steady-state.}
For the choice of the voltage/reactive power control $u_Q$, we set $u_Q=\overline u_Q$ where $\overline u_Q$ is a constant vector enforcing the setpoint for the voltage dynamics. The role of this setpoint will be made clear in Subsection \ref{subset.power.sharing}.
Then, the main result of this section is as follows:

\begin{thm}\label{t:convergance}
\color{black}{Suppose that the vectors \cter{$\theta^0\in \mathbb{T}^n$ and $\overline V\in \R^n$} are such that Assumption \ref{assum.feasibility} and condition \eqref{sconvex}, with $\overline \omega=\omega^\ast$,  hold. Let $u_P$ be given by \eqref{active.controller}, $u_Q=\overline u_Q  \in \R^n $, 
and $\overline u_P$  the optimal input \eqref{optimal.u_P}.
Then, the following statements hold:
\begin{enumerate}[(i)]
\item  The vector $(D^T\theta, \omega, V, \xi)$ with $(\theta, \omega, V, \xi)$ being a solution to \eqref{micro.grid}, \eqref{active.controller}, with the conventional droop controller \eqref{f-droop}, quadratic droop controller \eqref{f-Qdroop}, or reactive current controller \eqref{f-current}, locally\footnote{``locally" refers to the fact that the solutions are initlialized in a suitable neighborhood of $(\overline\theta, \overline\omega, \overline V, \overline\xi)$.}  converges to  the point $(D^T\theta^0, \overline\omega, \overline V, \overline\xi)$.
\medskip{}
\item  The vector $(D^T\theta, \omega, V, \xi)$ 
with $(\theta, \omega, V, \xi)$ being a solution to \eqref{micro.grid}, \eqref{active.controller}, with \cter{the} E-ARP controller \eqref{f-consensus}, locally converges to a point in the set
\begin{align*}
\{(D^T\theta, \omega, V, \xi) \mid \;& \omega=\omega^\ast, \xi=\overline u_P,\\ &P=\overline P, L_QK_Q Q=K_Q^{-1}\overline u_Q 
\}
\end{align*}
Moreover, {for all $t\ge 0$,  }
\[
{\mathds{1}^TK_Q^{-1}\; \boldsymbol{\ln}(V(t))=}
\mathds{1}^TK_Q^{-1}\; \boldsymbol{\ln}(V(0)).
\]
\end{enumerate}}
\end{thm}
\bigskip{} 
\begin{proof}
\New{First recall that $\varphi = E^T \theta$, $\overline \varphi=E^T \overline \theta$, and $D_1^T\overline \varphi=D^T\overline \theta=D^T\theta^0$ with $E^T= [I_{n-1}\; -\mathds{1}_{n-1}]$ and noting that $ED_1=D$.} By the compatibility property of the induced matrix norms, we have $\|\varphi(0)-\overline \varphi\|\le \|E^T \| \|\theta(0)-\overline \theta(0)\|$, thus showing that a choice of $\theta(0)$ sufficiently close to $\theta^0=\theta(0)$, returns an initial condition $\varphi(0)$ sufficiently close to $\overline \varphi$. 
We then consider a solution $(\theta,\omega,V,\xi)$ to the closed-loop system and express the solution into the new coordinates as $(\varphi,\omega,V,\xi)$.
Define the incremental storage function 
\be\label{inc-Lyap-active-controller}
\mathcal{C}(\xi)=\frac{1}{2}(\xi-\overline \xi)^T(\xi-\overline \xi).
\ee
Notice that $\overline \xi \in \im \mathds{1}$. Then
\[\ba{l}
\displaystyle\frac{d}{dt}\mathcal{C} =-(\xi-\overline \xi)^TL_P(\xi-\overline \xi)-(\xi-\overline \xi)^T  K_P^{-1}  (\omega-\overline \omega)\\
\quad =-(\xi-\overline \xi)^TL_P(\xi-\overline \xi)-
(u_P-\overline u_P)^T K_P^{-1}(\omega-\overline \omega).
\ea\] 

By \eqref{shift-co}, the time derivative of $\calS(\theta, \omega, V)$ is equal to that of $\calS(\varphi, \omega, V)$, with $\varphi$ obtained from \eqref{varphi}, namely (with \cter{\eqref{shift-co} in mind)}
\be\label{Nima-dt-equality}
\ba{rcl}
\displaystyle\frac{d}{dt}\calS(\theta, \omega, V) 
&=& \displaystyle\frac{d}{dt}\calS(\varphi, \omega, V).
\ea
\ee
Hence, from the proof of Theorem \ref{t:droop} we infer that 
\be\label{key.inequality}
\ba{rcl}
&& \displaystyle\frac{d}{dt}\calS(\varphi, \omega, V) \\[2mm]
&=& 
-( \omega-\overline \omega)^T  K_P^{-1} (\omega -\overline \omega) + ( \omega-\overline \omega)^T  K_P^{-1}(u_P-\overline u_P)
\\[2mm]
&&
-(\displaystyle\frac{\partial \calS}{\partial V})^T X(V) \displaystyle\frac{\partial \calS}{\partial V}+ (\displaystyle\frac{\partial \calS}{\partial V})^T Y(V)  (u_Q-\bar u_Q),
\ea
\ee
where $X(V)= T_Q^{-1} K_Q[V], T_Q^{-1}\textrm{ or } [V] K_Q L_Q  K_Q[V]$ and $Y(V)= T_Q^{-1}, T_Q^{-1} [V], [V]$ depending on the voltage controller adopted.

Observe that, by setting $u_Q=\overline u_Q$ and bearing in mind (\ref{key.inequality}), the equalities \eqref{dS-droop}, 
\eqref{dU-Qdroop}, \eqref{dU-current}, and \eqref{dU-consensus} can be written in a unified manner as 
\[\ba{l}
\displaystyle\frac{d}{dt}\calS(\varphi, \omega, V)=-(\omega-\overline \omega)^TK_P^{-1}(\omega-\overline \omega)\\
\qquad -\left(\fpart{\calS}{V}\right)^TX(V)\fpart{\calS}{V}+(\omega-\overline \omega)^TK_P^{-1}(u_P-\overline u_P)
\ea\]
where $X$ is a positive (semi)-definite matrix {\color{black} defined above}.  
Now taking $\calS+\mathcal{C}$ as the Lyapunov function, we have
\be\label{total.storage.dot}\ba{c}
\displaystyle \frac{d}{dt}\calS+\frac{d}{dt}\mathcal{C}=-(\omega-\overline \omega)^TK_P^{-1}(\omega-\overline \omega)\\[2mm]
-\left(\fpart{\calS}{V}\right)^TX(V)\fpart{\calS}{V}
-(\xi-\overline \xi)^TL_P(\xi-\overline \xi).
\ea\ee
By local strict convexity of $\calS+\mathcal{C}$ (thanks to \eqref{sconvex}), we can construct a forward invariant compact level set $\Omega$ around $(\overline \varphi, \overline \omega, \overline V,\overline \xi)$
and apply  LaSalle's invariance principle. Notice in particular that on this forward invariant set $V(t)\in \R^n_{>0}$ for all $t\ge 0$. Then the solutions are guaranteed to converge to the largest invariant set where
\be\label{conv.inv.set}
\ba{rcl}
\omega&=&\overline \omega\\[2mm]
0&=&L_P(\xi-\overline \xi)\\[2mm]
0&=&\left(\fpart{\calS}{V}\right)^TX(V)\fpart{\calS}{V}
\ea
\ee
\New{The first equality yields $\fpart{\mathcal{S}}{\omega}=0$ on the invariant set.} Recall that $\overline \xi\in \im \mathds{1}$. Hence, on the invariant set, $L_P\xi=0$ and thus $\xi=\gamma \mathds{1}$ for some $\gamma \in \R$. Note that, by \eqref{active.controller},  $\gamma$ has to be constant given the fact that $\omega=\omega^\ast$ and $L_P\xi=0$.
Also note that 
\[
u_P=K_P(D\Gamma(V)\boldsymbol{\sin} (D^T\theta^0)-P^\ast)
\]
on the invariant set. Multiplying both sides of the above equality by $\mathds{1}^TK_P^{-1}$ yields
 $\gamma \mathds{1}^T K_P^{-1} \mathds{1}=-\mathds{1}^TP^\ast$. 
Therefore,  $\xi=-\frac{\mathds{1}\mathds{1}^TP^\ast}{\mathds{1}^T K_P^{-1} \mathds{1}}$,  and on the invariant set, $u_P$ is equal to the optimal input $\overline u_P$ given by \eqref{optimal.u_P}. \color{black}{This also means that $\fpart{\mathcal{C}}{\xi}=0$.

Notice that any solution $(\varphi, \omega, V, \xi)$ on the invariant set satisfies
\[
0 =-T_{P}^{-1} K_P 
E\displaystyle\frac{\partial \mathcal{S}}{\partial \varphi}
 - T_P^{-2} K_P \frac{\partial \mathcal{S}}{\partial \omega} + T_{P}^{-1} \displaystyle\frac{\partial \mathcal{C}}{\partial \xi}.\\[2mm]
\]
Hence, evaluating the dynamics above on the invariant set yields 
$\frac{\partial \mathcal{S}}{\partial \varphi}=0$
noting that the matrix $E$ has full column rank.

Furthermore, by \eqref{dS-droop}, \eqref{dU-Qdroop}, and \eqref{dU-current}, the matrix $X(V)$ is positive definite for the droop controller, quadratic droop controller, and the reactive current controller. Hence, the third equality in \eqref{conv.inv.set} yields $
\frac{\partial \mathcal{S}}{\partial V}=0$
on the invariant set.

Therefore, the partial derivatives of $\mathcal{S}+\mathcal{C}$ vanish on the invariant set. Now, as the solution is evolving in a neighborhood where there is only one isolated minimum  $(\overline \varphi,\overline \omega,  \overline V, \overline \xi)$ of $\mathcal{S}+\mathcal{C}$, then the  invariant set only comprises such a minimum, and therefore convergence to the latter is guaranteed. This verifies the first statement of the theorem noting that the convergence of $\varphi$ to $\overline \varphi$ implies that of $D^T\theta$ to $D^T\theta^0$ by continuity and the equality $ED_1=D$.}

For the E-ARP controller, we have $X(V)=[V]K_Q L_Q K_Q [V]$ as evident from \eqref{dU-consensus}. Hence, by \eqref{dSdV--consensus} and the third equality in \eqref{conv.inv.set}, {on the invariant set} we obtain that 
\be\label{Q&barQ}
L_Q K_Q Q=L_Q K_Q\overline Q.
\ee
By \eqref{baruQ-consensus} and \eqref{Q&barQ}, the vector ${Q}$ satisfies on the invariant set
\be\label{tilde.Q.1}
L_Q K_Q{Q}=K_Q^{-1}\overline u_Q.
\ee
Notice that, for the E-ARP controller, we have so far shown that the solutions $(\varphi, \omega, V, \xi)$ converge to the set 
\begin{align*}
\mathcal{Q}:=\{(\varphi, \omega, V, \xi) \in \Omega &\mid \;\omega=\omega^\ast, \xi=\overline u_P,\\
 &P=\overline P, L_QK_Q Q=K_Q^{-1}\overline u_Q 
\}.
\end{align*} 
Next, we establish the convergence of trajectories to {\em a point} in $\mathcal{Q}$. To this end, we take the forward invariant set $\Omega$ sufficiently small such that  
\be\label{hessian10}
\frac{\partial^2 (\mathcal{S}+\mathcal{C})}{\partial (\varphi, \omega, V,\xi)^2}>0
\ee
for every $(\varphi,  \omega,  V,  \xi)\in \mathcal{Q}$. 
Note that this is always possible by \eqref{sconvex}  and continuity.
Observe that any solution $(\varphi, \omega, V, \xi)$ satisfies
\[\ba{rcl}
\dot \varphi &=& 
E^T K_{P}T_{P}^{-1} 
\displaystyle\frac{\partial \mathcal{S}}{\partial \omega}\\[2mm]
\dot \omega &=&-T_{P}^{-1} K_P 
E\displaystyle\frac{\partial \mathcal{S}}{\partial \varphi}
 - T_P^{-2} K_P \frac{\partial \mathcal{S}}{\partial \omega} + T_{P}^{-1} \displaystyle\frac{\partial \mathcal{C}}{\partial \xi}\\[2mm]
\dot V &=&- T_Q^{-1}X(V)\displaystyle\frac{\partial \mathcal{S}}{\partial V}\\
\dot \xi &=& -L_C \displaystyle\frac{\partial \mathcal{C}}{\partial \xi} - K_{P}T_{P}^{-1} 
\displaystyle\frac{\partial \mathcal{S}}{\partial \omega}.
\ea
\]
It is easy to see that every point of $\mathcal{Q}$ is an equilibrium of the system above, and by \eqref{hessian10} is  Lyapunov stable. In fact, by \eqref{hessian10}, the incremental storage function $\mathcal{S}+\mathcal{C}$ can be analogously defined with respect to any point in $\mathcal{Q}$ to establish Lyapunov stability by the inequality $\dot{\calS}+\dot{\mathcal{C}}\le 0$. Therefore, the positive limit set associated with any solution issuing from a point in $\Omega$ contains a Lyapunov stable equilibrium. It then follows by \cite[Proposition 4.7]{haddad2008nonlinear}\footnote{For the corrected version, see the errata and addenda in {http://soliton.ae.gatech.edu/people/whaddad/journal/Errata.pdf}} that this positive limit set is a singleton which proves the convergence to a point in $\mathcal{Q}$. This proves the claim in the second statement of the theorem given the relationship between $\theta$ and $\varphi$ variables exploited before. 

Finally, by \eqref{baruQ-consensus}, the E-ARP controller  can be written as
\[
\dot{V}=-[V]K_QL_Q K_Q  (Q-\overline Q). 
\] 
Hence, we have
\[
\frac{d}{dt}(\mathds{1}^T K_Q^{-1}\;\boldsymbol{\ln}V)=\mathds{1}^T K_Q^{-1} [V]^{-1}[V]K_Q L_Q K_Q (Q-\overline Q)=0,
\]
as $\mathds{1}^T L_Q = 0$, 
which proves that $\mathds{1}^T K_Q^{-1}\;\boldsymbol{\ln}(V)$ is a conserved quantity. 
\end{proof}
\medskip{}
{\color{black}
\begin{remark}
{\bf (Stability under feedforward control)}  When the input $u_P$ is set to the optimal feedforward input $\overline u_P$, rather than being generated by the feedback controller \eqref{active.controller}, the closed-loop system takes the form 
\[\ba{rcl}
\dot \varphi &=& 
E^T K_{P}T_{P}^{-1} 
\displaystyle\frac{\partial \mathcal{S}}{\partial \omega}\\[2mm]
\dot \omega &=&-T_{P}^{-1} K_P 
E\displaystyle\frac{\partial \mathcal{S}}{\partial \varphi}
 - T_P^{-2} K_P \frac{\partial \mathcal{S}}{\partial \omega}\\[2mm]
\dot V &=&- T_Q^{-1}X(V)\displaystyle\frac{\partial \mathcal{S}}{\partial V}.
\ea
\]
The same arguments as in the proof  above then show that solutions to this closed-loop system locally converge to the equilibrium point $(\overline \varphi,\overline \omega,  \overline V)$. Hence, the stability of this equilibrium is an intrinsic property of the closed-loop system obtained setting $u_P=\overline u_P$, $u_Q=\overline u_Q$. The adoption of the distributed integral controller \eqref{active.controller} is required to overcome the lack of knowledge of $\overline u_P$, which depends on global parameters. 
\end{remark}
}
\subsection{Power sharing }\label{subset.power.sharing} 
Theorem \ref{t:convergance} portrays the asymptotic behavior of the microgrid models discussed in this paper, {\color{black} namely} {\em frequency regulation} and {\em voltage stability}. {\color{black} In addition,} {\em optimal active power sharing} for the coupled nonlinear microgrid model \eqref{micro.grid} {\color{black} is achieved if the droop coefficients $K_P$ are suitably chosen. In fact,} 
substituting \eqref{optimal.u_P} into \eqref{syncw-ss} {\color{black} yields}
\[
\overline P =  P^* - K_P^{-1}  \mathds{1}\frac{\mathds{1}^TP^\ast}{\mathds{1}^T K_P^{-1} \mathds{1}},
\]
or, component-wise,  
\[
\overline P_i =  P_i^* - (k_{P})_i^{-1}  \frac{\mathds{1}^TP^\ast}{\mathds{1}^T K_P^{-1} \mathds{1}},
\]
where $K_P=[k_P]$. In  case  droop coefficients are selected proportionally (\cite{simpson,dorfler.et.al.arxiv14, 
andreasson.acc14,mb:cdp:st:mtns14, ST:MB:CDP:AUT15}), i.e.
\[
(k_{P})_i P_i^*  = (k_{P})_j P_j^*,
\] 
for all $i,j$, we conclude that 
\[
(k_{P})_i \overline P_i  = (k_{P})_j \overline P_j,
\]
which accounts for the desired proportional active power sharing based on the diagonal elements of $K_P$ as expected.

Next, we take a closer look at other consequences and implications of Theorem \ref{t:convergance} for different voltage dynamics.
\\[2mm]
\noindent 1)\;{\em Conventional droop controller}\/: 
The vectors of voltages and reactive powers converge to $\overline V$ and $\overline Q$ satisfying
\be\label{e:rev-droop-sharing}
K_Q \overline Q + \overline V= \overline u_Q  
\ee
which yields
\be\label{e:rev-droop-sharing2}
\frac{(k_Q)_i\overline Q_i+\overline V_i}{(k_Q)_j\overline Q_j+\overline V_j}=\frac{(\overline u_Q)_i}{(\overline u_Q)_j}.
\ee
\claudio{This results in  partial voltage regulation and reactive power sharing for the droop controlled inverters.
In fact, for small values of $k_Q$, $\overline u_Q$ regulates the voltages following \eqref{e:rev-droop-sharing}. On the other hand, if the elements of $k_Q$ are sufficiently large, reactive power is shared according to the elements of $\overline u_Q$ as given by \eqref{e:rev-droop-sharing2}.
\claudio{This tunable tradeoff between voltage regulation and reactive power sharing is consistent with the findings of \cite{simpson.industrial}.} 
}
\\[2mm]
\noindent 2)\;{\em Quadratic droop controller}\/: The vector of voltages and reactive power converge to $\overline V$ and $\overline Q$ with
\[
K_Q[\overline V]^{-1}\overline Q + \overline V=\overline u_Q. 
\]
This implies that
\[
\frac{(k_Q)_i\overline Q_i+\overline V_i^2}{(k_Q)_j\overline Q_j+\overline V_j^2}=\frac{(\overline u_Q)_i}{(\overline u_Q)_j}
\]
which again results in a partial \claudio{voltage regulation and reactive power sharing by an appropriate choice of $k_Q$ and $\overline u_Q$.}
Moreover, in this case, the voltage variables at steady-state are explicitly given by
\[
\overline V=(I+K_Q\mathcal{A} (\boldsymbol{\cos}(D^T\theta^0)))^{-1}  \overline u_Q.
\]
\noindent 3)\;{\em Reactive current controller}:\/ In this case, we have
\[
[\overline V]^{-1}\overline Q=\overline u_Q
\]
which results in
\[
\frac{\;\;\frac{\overline Q_i}{\overline V_i}\;\;}{\frac{\overline Q_j}{\overline V_j}}=\frac{(\overline u_Q)_i}{(\overline u_Q)_j}=
(\frac{\overline V_j}{\overline V_i})\;(\frac{\overline Q_i}{\overline Q_j}).
\]
The first equality provides the exact reactive current sharing, whereas the second equality can be interpreted as a mixed voltage and reactive power sharing condition.  
Moreover, the voltage variables at steady-state are given by
\[
\overline V=\mathcal{A}^{-1}(\boldsymbol{\cos}(D^T\theta^0))\overline u_Q.
\]
\\[2mm]
\noindent 4)\;{\em Exponentially-scaled averaging reactive power controller}\/:
%
In this case, 
the exact reactive power sharing can be achieved as evident from the \new{second} statement of Theorem \ref{t:convergance}, \new{with $\overline u_Q=0$.} In particular, by equality (\ref{tilde.Q.1}) with $\overline u_Q=0$ we obtain that 
%
\[
{Q}=\alpha K_Q^{-1}\mathds{1}
\]
for some $\alpha \in \R$. 
Multiplying both sides of the above equality by $\mathds{1} ^T$ yields
\[
\alpha=\frac{\mathds{1}^T Q}{\mathds{1}^T K_Q^{-1} \mathds{1}}.
\]
Clearly, $\alpha >0$, by definition of $Q$ and  as the matrix $\mathcal{A}$ is positive definite.
Therefore, as a consequence of Theorem \ref{t:convergance}, the vector of reactive power converges to a constant vector $\widetilde{Q}\in \R^n_{>0}$ where
\be\label{reactive-sharing}
\new{(k_Q)_i \widetilde{Q}_i = (k_Q)_j \widetilde{Q}_j,} 
\ee
which guarantees {proportional} reactive power sharing according to the elements of $k_Q$ as desired. 
Notice that the quantity \new{$\mathds{1}^T K_Q^{-1}\;\boldsymbol{\ln}V$} is a conserved quantity in this case. Hence, the point  of convergence for the voltage variables is primarily determined by the initialization $V(0)$.

\subsection{Lossy lines}\label{subs.lossy}

Under appropriate conditions, 
{\color{black}the stability of the system dynamics under the various controllers}
are preserved in the presence of lossy transmission lines that are homogeneous, namely whose impedences $Z_{ij}$ equal $|Z_{ij}|{\rm e}^{\sqrt{-1} \phi}$, with $\phi \in [0, \frac{\pi}{2}]$. Consistently, 
shunt components at the buses that are a series interconnection of a resistor and an inductor whose impedance is $\cter{\hat{r}}_{ii}+\sqrt{-1} \cter{\hat{x}}_{ii}$ are considered. 
Assuming homogeneity of the shunt elements, i.e.  $\cter{\hat{r}}_{ii}+\sqrt{-1} \cter{\hat{x}}_{ii}= \sqrt{{\cter{\hat{r}}_{ii}^2+\cter{\hat{x}}_{ii}^2}}{\rm e}^{\sqrt{-1}\arctan\frac{\cter{\hat{x}}_{ii}}{\cter{\hat{r}}_{ii}}}=|Z_{ii}| {\rm e}^{\sqrt{-1}\arctan\phi}$, where $\phi = \arctan\frac{\cter{\hat{x}}_{ii}}{\cter{\hat{r}}_{ii}}$ for all $i$, 
{\color{black} routine derivations  (see e.g.~\cite{zhong, muenz}) show that} 
the total active and reactive power $P_i^\ell, Q_i^\ell$ ``{\color{black}exchanged}" 
by the inverter $i$
\claudio{in the lossy network is equal to
\be\label{P.Q.lossy}
\begin{bmatrix}
P_i^\ell \\
Q_i^\ell 
\end{bmatrix}
=\Phi(\phi)
\begin{bmatrix}
P_i \\
Q_i 
\end{bmatrix}
\ee
where
\[
\Phi(\phi)=
\begin{bmatrix}
\sin \phi & \cos \phi  \\
-\cos \phi & \sin \phi
\end{bmatrix},
\]
%
%
and $P_i$, $Q_i$, have the same expressions as in \eqref{active} and \eqref{reactive}. Hence, the matrix $\Phi(\phi)$ will modify the expressions of the active and reactive power exploited previously, and thus the frequency and voltage dynamics of the inverters will be changed accordingly, \claudio{disrupting the convergence of the solutions}. A natural way to counteract this modification is to exploit the inverse of 
$\Phi(\phi)$  and \claudio{use $P^\ell \sin \phi-Q^\ell \cos\phi$ and $P^\ell \cos \phi+Q^\ell \sin\phi$, with $P^\ell=\col(P^\ell_i)$ and $Q^\ell=\col(Q^\ell_i)$, in \eqref{micro.grid} instead of  $P$ and $Q$, respectively.} 
In this way, the {\em lossless} expressions of $P_i$ and $Q_i$ as in \eqref{active} and \eqref{reactive} will be recovered. Notice that, however, the implementation of these controllers requires the knowledge of the parameter $\phi$ which is assumed to be available. An interesting special case is obtained for $\phi=0$ meaning that the network is purely resistive.  In that case, in \eqref{micro.grid} $P$ should be replaced by $-Q^\ell$, and $Q$ by $P^\ell$, which is consistent with the use of droop controllers in resistive networks (see e.g. \cite[Sec. II.A]{Brabandere}).}

\claudio{As a result of the adaptation above, the same conclusions\footnote{In these conditions, whenever relevant, the negative of the susceptances $\hat B_{ii}, B_{ij}$ should be replaced by $|\cter{\hat{Z}}_{ii}|^{-1}, |Z_{ij}|^{-1}$.} as in  Theorem \ref{t:convergance} holds for the lossy network with modified inverter dynamics. Notice, however, that the actual active power $P^\ell$ will no longer be optimally shared in a lossy network with the conventional droop controller \eqref{f-droop}, quadratic droop controller \eqref{f-Qdroop}, or the reactive current controller \eqref{f-current}. Remarkably,} 
{\color{black}{in the case of the E-ARP controller, one can additionally prove that both active 
as well as reactive power sharing continues to hold. Because of its importance, the result is formalized below.
\begin{proposition}\label{t:convergance:lossy}
Suppose that Assumption \ref{assum.feasibility} with $f(V,Q,u_Q)={- [V]}K_QL_Q K_Q Q$ and condition \eqref{sconvex}, with $\overline \omega=\omega^\ast$ and $\hat B_{ii}, B_{ij}$ replaced by $|\cter{\hat{Z}}_{ii}|^{-1}, |Z_{ij}|^{-1}$, respectively, hold. 
Then the vector $(D^T \theta, \omega, V, \xi)$ with $(\theta, \omega, V, \xi)$  a solution of
\be\label{micro.grid.consensus.lossy}
\ba{rl}
\dot \theta &= \omega\\
T_P\dot \omega &= -(\omega-\omega^*) - K_P(P^\ell \sin \phi-Q^\ell \cos\phi-P^*) +u_P\\ 
\dot V &= -[V]K_QL_Q K_Q (P^\ell \cos \phi+Q^\ell \sin\phi) \\
\ea\ee
and $u_P$ given by \eqref{active.controller},  locally converges to a point in the set 
\begin{align*}
\{(D^T\theta, \omega, V, \xi) \mid \;& \omega=\omega^\ast, \xi=\overline u_P,\\ &P=\overline P, L_QK_Q Q=\boldsymbol{0}
\}.
\end{align*}
Moreover,  
${\mathds{1}^TK_Q^{-1}\; \boldsymbol{\ln}(V(t))=}\mathds{1}^TK_Q^{-1}\; \boldsymbol{\ln}(V(0))$, for all $t\ge 0$.
Finally, $P^\ell, Q^\ell$ converge to constant vectors $\overline P^\ell, \overline Q^\ell$ that
satisfy
\be\label{power.sharing.lossy}\ba{rcl}
(k_P)_i \overline P_i^\ell &=& (k_P)_j \overline P_j^\ell\\
(k_Q)_i  \overline Q_i^\ell &=& (k_Q)_j  \overline Q_j^\ell,
\ea\ee
provided that 
\be\label{same.droop}
\displaystyle\frac{(k_{P})_i}{(k_{P})_j}=\frac{(k_{Q})_i}{(k_{Q})_j},\quad \forall i,j.
\ee
\end{proposition}}

\begin{proof}
As remarked above, the convergence of the solutions is an immediate consequence of Theorem \ref{t:convergance}. Thus, we only focus on the power sharing property.

By condition \eqref{same.droop} and relation (\ref{P.Q.lossy}) at steady state,
\[\ba{rcl}
\overline P_i^\ell &=& \overline P_i \sin \phi +\widetilde Q_i\cos \phi\\
&=& \displaystyle\frac{(k_{P})_j}{(k_{P})_i}\overline P_j\sin \phi  +\frac{(k_{Q})_j}{(k_{Q})_i}\widetilde Q_j \cos \phi\\
&=& \displaystyle\frac{(k_{P})_j}{(k_{P})_i} \overline P_j^\ell.
\ea\]

Similarly, for the reactive power
 \[\ba{rcl}
\overline Q_i^\ell & = &-\overline P_i \cos\phi +\widetilde Q_i\sin\phi \\
 &=& -\displaystyle\frac{(k_{P})_j}{(k_{P})_i}\overline P_j \cos\phi +\frac{(k_{Q})_j}{(k_{Q})_i}\widetilde Q_j\sin\phi\\
 &=& \displaystyle\frac{(k_{Q})_j}{(k_{Q})_i}(-\overline P_j \cos\phi +\widetilde Q_j\sin\phi)
\\
 &=& \displaystyle\frac{(k_{Q})_j}{(k_{Q})_i}       \overline Q_j^\ell.
\ea\]
\end{proof}
}
{\color{black}
\subsection{Dynamic extension} Another interesting feature is that thanks to the incremental passivity property the static controller $u_Q=\overline u_Q$ can be extended to a dynamic controller. \nima{By Theorem \ref{t:droop} and keeping in mind Definition \ref{d:dissipativity} together with \eqref{output} and \eqref{input}, the incremental input-output pair of the voltage dynamics appears in the time derivative of the storage function $\calS$ as
\be\label{e:io-dynamic-diss}
(\fpart{S}{V}-\barfpart{S}{V})^TT_Q^{-1}R_2(u_Q-\overline u_Q)
\ee
where $R_2$ is the lower diagonal block of $R$ in Theorem \ref{t:droop}.}
Clearly this cross term is vanished by applying the feedforward input $u_Q=\overline u_Q$. But an alternative way to compensate for this term is to introduce the dynamic controller\nima{
\be\label{uQ-dynamic}
\ba{rcl}
T_Q\dot{\lambda}&=&-R_2\fpart{\calS}{V}\\
u_Q&=&K_\lambda \lambda
\ea
\ee
for some positive definite matrix $K_\lambda$. Notice that that the controller above is decentralized for a diagonal matrix $K_\lambda$.}
Then, denoting the steady state value of $\lambda$ by $\overline \lambda$, the incremental storage function \nima{$\mathcal{C}_Q(\lambda)=\frac{1}{2}(\lambda-\overline \lambda)^TK_\lambda(\lambda-\overline \lambda)$ satisfies
\be\label{e:diss-dyn}
\frac{d}{dt}\mathcal{C}_Q=-(u_Q-\overline u_Q)^TT_Q^{-1}R_2\fpart{\calS}{V}
\ee
Note that
\[
\fpart{\calS}{V}=\fpart{S}{V}-\barfpart{S}{V}.
\]
Therefore, \eqref{e:diss-dyn} coincides with the negative of \eqref{e:io-dynamic-diss}, and thus the same convergence analysis as before can be constructed based on the storage function $\calS+\mathcal{C}+\mathcal{C}_Q$}. Consequently, the result of Theorem \ref{t:convergance} extends to the case of dynamic voltage/reactive power controller \eqref{uQ-dynamic}. 
\nima{For illustration purposes, below we provide the exact expression of the controller above in case of the conventional droop controller:
\be
\nonumber
\ba{rcl}
T_Q\dot{\lambda}&=&-[V]^{-1}K_Q^{-1}(K_Q(Q-\overline Q)+V-\overline V)\\
u_Q&=&K_\lambda\lambda
\ea
\ee
which by setting $K_\lambda=K_Q$ reduces to
\[
T_Q\dot{u}_Q=-[V]^{-1}(K_Q(Q-\overline Q)+V-\overline V)
\]
Note that here the constant vectors $\overline V$ and $\overline Q$ are interpreted as the setpoints of the dynamic controller. It is easy to see that this controller rejects any unknown constant disturbance entering the voltage dynamics \eqref{f-droop}. Other possible advantages of these dynamic controllers require further investigation, which is postponed to a future research.}

}

\section{Conclusions}\label{sec.VI}
We have presented a systematic design of incremental Lyapunov functions for the analysis and the design of network-reduced models of microgrids. Our results encompass existing ones and lift restrictive conditions, thus providing a powerful framework where microgrid control problems can be naturally cast. The method deals with the fully nonlinear model of microgrids and no linearization is carried out. 

Two major extensions can be envisioned. The first one is the investigation of similar techniques for network-preserved models of microgrids. Early results show that this is feasible and will be further expanded in a follow-up publication. 
The second one is how to use the obtained incremental passivity property to interconnect the microgrid with dynamic controllers and obtain a better understanding of voltage control. Examples of these controllers are discussed in \cite{simpson.industrial} but many others can be proposed and investigated. 

A more general question is how the set-up we have proposed can be extended to deal with {other control problems that are formulated} in the microgrid literature. Furthermore, the proposed controllers exchange information over a communication network and would be interesting to assess the impact of the communication layer on the results. In that regard, the use of Lyapunov functions is instrumental in advancing such research, since powerful Lyapunov-based techniques for the design of complex networked cyber-physical systems are already available (see e.g.~\cite{CDP:RP:CDC14}). 

\emod

\appendix
{\it Proof of Proposition \ref{p1}. }
For the sake of notational simplicity, in this proof we omit the bar from all $V, {\varphi}$. { 
\claudio{
Clearly, the Hessian \eqref{Hessian-3*3} is positive definite if and only if \eqref{Hess-1} 
holds.
%
The latter is true if and only the matrix $M$ below
\[
\begin{bmatrix}
\Gamma(V) [\boldsymbol{\cos}(D_1^T\varphi)] &[\boldsymbol{\sin}(D_1^T\varphi)]  \Gamma(V)  |D|^T  [V]^{-1} \\[2mm]
[V]^{-1} |D| \Gamma(V) [\boldsymbol{\sin}(D_1^T\varphi)] & \mathcal{A}(\boldsymbol{\cos}(D_1^T\varphi))+\displaystyle\frac{\partial^2 H}{\partial V^2}
\\[2mm]
\end{bmatrix}
\]
is positive definite. In fact recall that the matrix in \eqref{Hess-1} can be written as the product 
\[
\begin{bmatrix}
D_1 & \boldsymbol{0}\\
\boldsymbol{0} & I
\end{bmatrix}
M
\begin{bmatrix}
D_1^T & \boldsymbol{0}\\
\boldsymbol{0} & I
\end{bmatrix},
\]
and our claim descends from $D_1 D_1^T$ being nonsingular, the latter holding for $D_1 D_1^T$ is the principal submatrix of the Laplacian of a connected graph.
}
Furthermore, note that by assumption  $\Gamma(V) [\boldsymbol{\cos}(D_1^T\varphi)]$ is nonsingular. 
Then the  Hessian  is positive definite, or equivalently (\ref{Hess-1}) holds,  if and only if $\Gamma(V) [\boldsymbol{\cos}(D_1^T\varphi)]$ and 
\[\ba{c}
\Psi(D_1^T\varphi, V) := \mathcal{A}(\cos(D_1^T\varphi))+[h(V)] -[V]^{-1} |D| \Gamma(V)  \\  
~[\boldsymbol{\sin}(D_1^T\varphi)]^2 [\boldsymbol{\cos}(D_1^T\varphi)]^{-1}  |D|^T  [V]^{-1}> 0. 
\ea\]
Introduce the diagonal weight matrix, { where $\eta=D_1^T\varphi$,}  
\[
W(V, \eta):= \Gamma(V) [\boldsymbol{\sin}(\eta)]^2  [\boldsymbol{\cos}(\eta)]^{-1}.
\]
For each $k\sim \{i,j\}\in E$, its $k$th diagonal element is
\[
W_k(V_i, V_j , \eta_{k}):= B_{ij} V_i V_j \frac{\sin^2 (\eta_{k})}{\cos(\eta_{k})}.
\]
Furthermore, it can be verified that 
\[\ba{l}
\left[|D| \Gamma(V) [\boldsymbol{\sin}(\eta)]^2  [\boldsymbol{\cos}(\eta)\right]^{-1}  |D|^T]_{ij}
\\[3mm]
\qquad =\left\{
\ba{lll}
\displaystyle\sum_{k\sim \{i,\ell\}\in E} B_{i \ell} V_i V_\ell \frac{\sin^2 (\eta_{k})}{\cos(\eta_{k})} & \textrm{if} & i=j\\
B_{ij} \displaystyle V_i V_j \frac{\sin^2 (\eta_{k})}{\cos(\eta_{k})} & \textrm{if} & i\ne j,
\ea
\right.
\ea\]
from which 
\[\ba{l}
\left[[V]^{-1}  |D| \Gamma(V) [\boldsymbol{\sin}(\eta)]^2  [\boldsymbol{\cos}(\eta)\right]^{-1}  |D|^T [V]^{-1} ]_{ij} 
\\[3mm]
\qquad =\left\{
\ba{lll}
\displaystyle\sum_{k\sim \{i,\ell\}\in E} B_{i \ell} \frac{V_\ell}{V_i} \frac{\sin^2 (\eta_{k})}{\cos(\eta_{k})} & \textrm{if} & i=j\\
{\color{black}
\displaystyle B_{ij} \frac{\sin^2 (\eta_{k})}{\cos(\eta_{k})} 
}
& \textrm{if} & i\ne j.
\ea
\right.
\ea\]
On the other hand,
\[\ba{l}
\left[\mathcal{A}(\cos(\eta))+[h( V)]\right]_{ij}
\\[3mm]
\qquad=
\left\{
\ba{lll}
\hat B_{ii} + \displaystyle\sum_{k\sim \{i,\ell\}\in E} B_{i \ell} + h_i( V_i),
& \textrm{if} & i=j\\
- B_{i j}  \cos(\eta_{m}),
 & \textrm{if} & i\ne j
\ea
\right.
\ea\]
\cter{with} $m\sim\{i,j\}\in E$, \cter{$i\ne j$}.  
Suppose that each diagonal entry of matrix $\Psi(\eta, V)$ is positive, that is for each $i=1,2,\ldots, n$,
\[
\ba{l}
m_{ii}
=\hat B_{ii} + \displaystyle\sum_{\ell=1, \ell\ne i}^n B_{i \ell} + h_i(V_i)- \displaystyle\sum_{k\sim \{i,\ell\}\in E} B_{i \ell} \frac{V_\ell}{V_i} \frac{\sin^2 (\eta_{k})}{\cos(\eta_{k})}\\
=\hat B_{ii} + \displaystyle\sum_{k\sim \{i,\ell\}\in E} B_{i \ell} 
\left(
1-\frac{V_\ell}{V_i} \frac{\sin^2 (\eta_{k})}{\cos(\eta_{k})} 
\right)
+ h_i(V_i) 
>0. 
\ea\]
Notice that this holds true because of condition (\ref{cond2}).
Assume also that, for each $i=1,2,\ldots, n$,  
\[\ba{rcl}
m_{ii} &>&  \displaystyle\sum_{k\sim \{i,\ell\}\in E} B_{i \ell}  \left|   \cos(\eta_{k})  
{\color{black} + \frac{\sin^2 ({\color{black}\eta}_{k})}{\cos({\color{black}\eta}_{k})}
}
\right|\\
&=& 
\displaystyle
\sum_{
{k\sim\{i, \ell\}\in E}} B_{i \ell} \;
{
\claudio{ \sec(\overline \eta_k)}
}
, 
\ea\]
which is condition (\ref{cond2}). 
Then by Gershgorin theorem all the eigenvalues of the matrix $\Psi(\eta, V)$ have strictly positive real parts and the Hessian is positive definite.

\bigskip

\bigskip{}
\new{{\it Proof of Lemma \ref{inertia}. }
An incremental model of the dynamical system \eqref{micro.grid} with respect to a synchronous solution can be written as follows 
\be\label{micro.grid.incremental}
\ba{rcl}
\displaystyle \frac{d}{dt} (\theta-\overline \theta) &=& (\omega-\overline \omega)\\[2mm]
T_P\displaystyle \frac{d}{dt} (\omega-\overline \omega) &=& -(\omega-\overline \omega) - K_P(P-\overline P) +(u_P-\overline u_P)\\ [2mm]
T_Q\displaystyle \frac{d}{dt} (V-\overline V) &=& f(V,Q, u_Q)-f(\overline{V},\overline{Q}, \overline{u}_Q).
\ea
\ee
Recalling the equalities 
\eqref{dV--droop}, 
\eqref{dV--Qdroop}, \eqref{dV--current} 
\cter{and \eqref{shift-co}, the system in the $\varphi$ variables rewrites as}
\be\label{micro.grid-eta}
\ba{rcl}
\dot \varphi &=&  {\color{black}E^T}  T_P^{-1}K_P\displaystyle \frac{\partial \mathcal{S}}{\partial \omega} \\[2mm]
\dot \omega &=& -T_P^{-2}\cter{K_P}\displaystyle \frac{\partial \mathcal{S}}{\partial \omega} - T_P^{-1}K_P {\color{black}E} \displaystyle \frac{\partial \mathcal{S}}{\partial \varphi} +T_P^{-1}(u_P-\overline u_P)\\ 
\dot V &=& -T_Q^{-1}X(V)\displaystyle \frac{\partial \mathcal{S}}{\partial V}+Y(V) (u_Q -\overline u_Q)
\ea\ee
{\color{black}
where 
\[
E=\begin{bmatrix}
I_{n-1}\\
-\mathds{1}_{n-1}^T
\end{bmatrix}
\]
and
}
$X$ and $Y$ depend on the voltage dynamics, and are equal to the lower diagonal block of $W$ and $R$ in Theorem \ref{t:droop}, respectively. 
The Jacobian of the right-hand side at steady-state with $u_P=\overline u_P$, $u_Q=\overline u_Q$ is 
\[
\begin{small}
\begin{bmatrix}
0 & {\color{black}E}^T T_P^{-1}K_P\displaystyle \frac{\partial^2 \mathcal{S}}{\partial \omega^2}
& 0\\[4mm]
- T_P^{-1}K_P {\color{black}E}\displaystyle \frac{\partial^2 \mathcal{S}}{\partial \varphi^2} & -T_P^{-2}\cter{K_P}\displaystyle \frac{\partial^2 \mathcal{S}}{\partial \omega^2} & - T_P^{-1}K_P D\displaystyle \frac{\partial^2 \mathcal{S}}{\partial V \partial \varphi}\\[4mm]
-T_Q^{-1}X(V)\displaystyle \frac{\partial^2 \mathcal{S}}{\partial \varphi \partial V} & 0 & 
-T_Q^{-1}X(V)\displaystyle \frac{\partial^2 \mathcal{S}}{\partial V^2}
\end{bmatrix}_-
\end{small} 
\]
where we have used the fact that $ \frac{\partial \mathcal{S}}{\partial V}$ vanishes at steady state.
The matrix above, denoted by $\mathcal{F}$, can be decomposed in the port-Hamiltonian form as
\[
\mathcal{F}=(\mathcal{J}-\mathcal{R}{\color{black}(V))}\left.\frac{\partial^2 S}{\partial (\varphi, \omega, V)^2}\right|_-
\]
where
\be\label{calJ}
\mathcal{J}= \begin{bmatrix}
0
& {\color{black}E}^T T_P^{-1}K_P&0\\
- K_PT_P^{-1}{\color{black}E} & 0 & 0\\
0 & 0 & 0
\end{bmatrix}
\ee
and
\be\label{calR}
\mathcal{R}(V) = \begin{bmatrix}
0
& 0
& 0\\
0 & -T_P^{-2}\cter{K_P} & 0\\
0 & 0 & 
-T_Q^{-1}X(V)
\end{bmatrix}.
\ee
It is not difficult to observe that the matrix $\mathcal{J}-\mathcal{R}{\color{black}(V)}$ is nonsingular 
and then by leveraging inertia theorems for matrices, \cite{carlson-inertia}, \cite{ostrowski-inertia}, it follows that the matrix $\mathcal{F}$ possesses an eigenvalue in the open right half plane, and thus the equilibrium $(\overline \varphi, \overline \omega, \overline V)$ is unstable.
%
}

\bigskip

{\it Proof of Lemma \ref{lemma.cut}.} 
As before, we omit the bar from the variables. \cter{First observe that if  $\overline \eta=D_1^T \overline \varphi \not \in (-\frac{\pi}{2}, \frac{\pi}{2})^m$ then the top-left block of $M$, namely $\Gamma(V) [\boldsymbol{\cos}(D_1^T \varphi)]$, has a negative eigenvalue, and this implies the existence of a vector $v\ne 0$ such that $v^TM v<0$, thus showing that $M$ has  a negative eigenvalue. Thus, in the remaining of the proof we let  $\overline \eta\in (-\frac{\pi}{2}, \frac{\pi}{2})^m$. } 
We write the quadratic form as
\[\ba{rcl}
v^T M v &=& (v^{(1)})^T \Gamma(V)[\boldsymbol{\cos}(\cter{\eta})] v^{(1)} + \\
&& 2 (v^{(2)})^T [V]^{-1}|D|\Gamma(V)[\boldsymbol{\sin}(\cter{\eta})]v^{(1)}+\\
&& (v^{(2)})^T (\mathcal{A}(\boldsymbol{\cos}(\eta))+\displaystyle\frac{\partial^2 H}{\partial V^2}) v^{(2)},
\ea\]
\cter{where again $\eta= D_1^T \varphi$.}
Expanding the first term on the right-hand side above, it is obtained
\[
\ba{rcl}
&&(v^{(1)})^T \Gamma(V)[\boldsymbol{\cos}(\cter{\eta})] v^{(1)} \\
&=& \displaystyle\sum_{k\in E} v_k^2 \gamma_k(V)\cos(\eta_k),
\ea
\]
having used 
$\gamma_k(V)=V_i V_j B_{ij}$, \claudio{$k\sim\{i,j\}$}. \\
The second term can be written as 
\[\ba{rcl}
&&2 (v^{(2)})^T [V]^{-1}|D|\Gamma(V)[\boldsymbol{\sin}(\cter{\eta})]v^{(1)}\\
&=& 2\displaystyle \sum_{\cter{i\in \mathcal{I}}} v_{{\color{black}m}+i}V_i^{-1} \sum_{k\in E} |d_{ik}|v_k \gamma_k(V)\sin(\eta_k) \\
&=&  2 \displaystyle\sum_{k\in E, k\sim \{i,j\}} B_{ij} \sin(\eta_k) (v_{{\color{black}m}+i}V_j+ v_{{\color{black}m}+j} V_i) v_k,
\ea\]
whereas the third term becomes
\[\ba{l}
(v^{(2)})^T (\mathcal{A}(\boldsymbol{\cos}(\cter{\eta}))+\displaystyle\frac{\partial^2 H}{\partial V^2}) v^{(2)}=\\ 
\displaystyle\sum_{i\in \cter{\mathcal{I}}} (
(B_{ii}+h_i( V_i)) v_{{\color{black}m}+i}^2 
-
\hspace{-0.75cm}\sum_{j\in \cter{\mathcal{I}}, j\ne i, k\sim\{i,j\}} 
\hspace{-0.75cm}
B_{ij} \cos(\eta_k)  v_{{\color{black}m}+i} v_{{\color{black}m}+j})= \\
\displaystyle\sum_{i\in \cter{\mathcal{I}}} 
(B_{ii}+h_i( V_i)) v_{{\color{black}m}+i}^2 
-2\hspace{-0.5cm}\sum_{k\in E, k\sim\{i,j\}} 
\hspace{-0.5cm}
B_{ij} \cos(\eta_k)  v_{{\color{black}m}+i} v_{{\color{black}m}+j}.
\ea\]
Overall we have 
\[\ba{l}
v^T M v = \hspace{-0.5cm}
\displaystyle \sum_{k\in E, k\sim\{i,j\}} 
\hspace{-0.5cm}
B_{ij}  (v_k^2 V_i V_j \cos(\eta_k)+2\sin(\eta_k) (v_{{\color{black}m}+i}V_j +\\
v_{{\color{black}m}+j} V_i) v_k-2 \cos(\eta_k)  v_{{\color{black}m}+i} v_{{\color{black}m}+j})+\displaystyle\sum_{i\in \cter{\mathcal{I}}} 
(B_{ii}+h_i(V_i)) v_{{\color{black}m}+i}^2. 
\ea\]
\cter{To prove the existence of a negative eigenvalue of the matrix $M$, we look for a vector $v$ 
such that $v^T Mv<0$. The candidate $v={\rm col}(v^{(1)}, v^{(2)})$ is a vector whose first subvector $v^{(1)}\in \mathbb{R}^m$ is associated to the cut-set, i.e. $v^{(1)}_k=\pm 1$ if $k\in K$ and $0$ otherwise, and whose second subvector $v^{(2)}\in \mathbb{R}^n$ satisfies $v^{(2)}_i\ne 0$ if and only if  $i\in \mathcal{I}$ and node $i$ is adjacent to an edge $k$ in the cut-set $K$, i.e.~$k\sim\{i,j\}\in K$. Bearing in mind the quadratic form derived above, the inequality $v^T Mv<0$ writes as}
%
%
\[\ba{c}
v^T M v =\displaystyle\sum_{k\in \cter{K}, k\sim\{i,j\}} B_{ij}  (v_k^2 V_i V_j \cos(\eta_k)+2(\sin\eta_k) \\[2mm]
\cdot (v_{{\color{black}m}+i}V_j+ v_{{\color{black}m}+j} V_i) v_k-2 \cos(\eta_k)  v_{{\color{black}m}+i} v_{{\color{black}m}+j})+\\[2mm]
 \displaystyle\sum_{i\in \cter{\mathcal{I}}: k\sim \{i,j\}, k\in \cter{K}} 
(B_{ii}+h_i( V_i)) v_{{\color{black}m}+i}^2 <0. 
\ea\]
For any $k\sim\{i,j\}\in K$, 
if one chooses
\[
v_{\cter{m}+i}= V_i v_k \overline{v}_{{\color{black}m}+i},\quad   v_{{\color{black}m}+j}= V_j v_k\overline v_{{\color{black}m}+j}
\]
with $\overline v_{{\color{black}m}+i}, \overline v_{{\color{black}m}+j}$ to be designed later, 
the previous inequality  is equivalent to
\[\ba{c}
\displaystyle\sum_{k\in K, k\sim \{i,j\}}  B_{ij} V_i V_j (\cos(\eta_k)- \\[2mm]
2 \cos(\eta_k) \overline v_{{\color{black}m}+i} \overline v_{{\color{black}m}+j}
+ 
2  \sin(\eta_k) (\overline v_{{\color{black}m}+i}+ \overline v_{{\color{black}m}+j} ) ) +\\[2mm]
\displaystyle\sum_{i\in \cter{\mathcal{I}}: k\sim \{i,j\}, k\in K} 
(B_{ii}+h_i( V_i)) V_i^2 \overline v_{{\color{black}m}+i}^2 <0.
\ea\]
Notice now that for any pair $i,j$ such that $k\sim \{i,j\}$ with $k\in K$, 
\[
(B_{ii}+h_i( V_i)) V_i^2 +(B_{jj}+h_j( V_j)) V_j^2 
\le 
\beta_k(V_i, V_j) B_{ij} V_i V_j, 
\]
which returns, by setting  $\overline v_{{\color{black}m}+i}= \overline v_{{\color{black}m}+j}=\overline v_k$, 
%
%
\[\ba{c}
\hspace{-0.15cm} \displaystyle \sum_{k\sim \{i,j\}\in K} \hspace{-0.5cm} B_{ij} V_i V_j (\cos(\eta_k)
+  
\displaystyle 2 \sin(\eta_k) \overline v_k +  \beta_k(V_i, V_j)  \overline v_k^2)<0,
\ea\]
\cter{where we have exploited the fact that $-2 \cos(\eta_k) \overline v_{{\color{black}m}+i} \overline v_{{\color{black}m}+j}=- 2 \cos(\eta_k) \overline v_k^2<0$ for all $k$ since $\eta_k \in (-\frac{\pi}{2}, \frac{\pi}{2})$. }
Hence, there exist $\overline v_k${'s} such that the inequality is satisfied if \claudio{ the discriminants are} positive, i.e.
$\sin(\eta_k)^2  - \beta_k(V_i, V_j)  \cos(\eta_k)>0$ \claudio{for all $k\in K$}. 
This ends the proof.

\medskip

\cter{{\it Proof of Proposition \ref{prop.instability}.}  Under the given assumptions, by Lemma \ref{lemma.cut} there exists a vector $v=(v^{(1)}, v^{(2)})\ne 0$ such that $v^T Mv <0$, where $v^{(1)}$ is the vector associated to the cut-set. Hence, it belongs to the cut-space, namely the column space of $D^T$ or equivalently  to the one of $D_1^T$. As a result, bearing in mind \eqref{Hessian-3*3}, \eqref{Hess-1}, the inequality $v^T Mv <0$ implies $w^T \left.\frac{\partial^2 S}{\partial (\varphi, \omega, V)^2}\right|_- w <0$ for some $w\ne 0$. In view of the symmetry of the Hessian, this in turn implies that the Hessian has a negative eigenvalue, thus proving the instability of $(\overline \varphi, \overline \omega, \overline V)$ by Lemma \ref{inertia}.} 


\end{document}